\documentclass{imsart}
\RequirePackage{natbib}
\usepackage{amsmath,amssymb,amsthm}
\usepackage{graphics,epsfig}
\usepackage{hyperref}
\usepackage{natbib}
\usepackage{color}
\usepackage{graphicx}
\usepackage{caption}
\usepackage{subcaption}
\usepackage{float}
\usepackage{dsfont}
\usepackage{mathrsfs}
\usepackage{multirow}
\usepackage{comment}
\usepackage{bm}

\def \qq {\boldsymbol{q}}

\def \ra {\rightarrow}

\def \E {\mathbb{E}}

\def \a {\alpha}
\def \be {\beta}

\newtheorem{example}{\bf Example}
\newtheorem{exm}[example]{\bf Example}

\newtheorem{definition}{\bf Definition}
\newtheorem{defn}[definition]{\bf Definition}

	\newtheorem{theorem}{\bf Theorem}
	
	\newtheorem{prop}{\bf Proposition}
	\newtheorem{lem}[theorem]{\bf Lemma}
	\newtheorem{as}{\bf Assumption}
	
	\newtheorem{cor}[theorem]{\bf Corollary}
	\newtheorem{coro}[theorem]{\bf Corollary}
	
\renewcommand{\epsilon}{\varepsilon}

\begin{document}

\begin{frontmatter}

\title{Scoring a Goal optimally in a Soccer game under Liouville-like quantum gravity action}
\runtitle{Soccer game}

\begin{aug}

\author{\fnms{Paramahansa} 
	\snm{Pramanik}
	\ead[label=e1]{ppramanik@southalabama.edu}}
\and
\author{\fnms{Alan M.} 
	\snm{Polansky}
	\ead[label=e2]{polansky@niu.edu}}
	
\runauthor{P. Pramanik and A. M. Polansky}

\affiliation{Northern Illinois University}

\address{Department of Mathematics and Statistics, University of South Alabama\\Department of Statistics and Actuarial Science, Northern Illinois University }
\end{aug}
  
\begin{abstract}
In this paper we present a new theoretical model to find out an optimal weight associated with a soccer player under the presence of a stochastic goal dynamics by using Feynman path integral method, where the action of a player is on $\sqrt{8/3}$-Liouville Quantum Gravity surface. Before determine the optimal weight we first establish an Infinitary logic which can deal with infinite variables on the strategy space then, a quantum formula of this logic has been developed and finally, based on this we are able show the existence of a Lefschetz-Hopf fixed point of this game. As in a competitive tournament, all possible standard strategies to score goals are known to the opposition team, a player's action is stochastic in nature and they would have some comparative advantage to score goals. Furthermore, conditions like uncertainties due to rain, dribbling and passing skill of a player, whether the match is a day match or a day-night match, advantages of having home crowd and asymmetric information of action profiles have been considered on the way to determine the optimal weight. 
\end{abstract}



\end{frontmatter}

\section{Introduction}
Soccer is one of the most popular sports of the world because of its simplistic rules. Three major tournaments of this game are The World Cup, The Euro Cup and The Copa America out of them The World Cup is the most popular tournament. Following \cite{santos2014} we know, after early $1990$s F\'ed\'eration Internationale de Football Association (FIFA) had been worried about that all the teams become defensive in terms of scoring goals. As a result, total number of goals have been falling which eventually leads to a fall in interest on this game. If we  consider $2006$ men's soccer World champion Italy and $2010$ champion Spain, they only allowed two goals in their entire seven matches in the tournament \citep{santos2014}. On March $17$, $1994$ in USA Today, FIFA clearly specified its objective as encourage attacking and high scoring match. Since then we can see some changes in objectives of some teams like the $2014$ World Cup semifinal between Brazil and Germany where Germany scores five goals in the first half. Furthermore, in $2018$ men's soccer World cup we see France chose some offensive strategies. Furthermore, if we look at the French national team of $2018$ men's World cup, we see there are no big names (like Lionel Messi, Cristiano Ronaldo, Romelu Lukaku etc.) and their average age is lower compare to other big teams. Therefore, all the players played without any pressure to win the World Cup and, if some team has a big name, that means that player is in the game for a long time and opposition teams have strategies to stop him to score goals. 

In this paper we present a new theoretical approach to find out an optimal weight associated with a soccer player under the presence of a stochastic goal dynamics by using Feynman path integral method, where the actions of every player of both the teams are on $\sqrt{8/3}$-Liouville Quantum Gravity (LQG) surface \citep{feynman1949,fujiwara2017,pramanik2020,pramanik2020opt,pramanik2021consensus,pramanik2021effects,pramanik2021s}. This surface is continuous but not differentiable. Furthermore, at $\sqrt{8/3}$ this behaves like a Brownian surface \citep{miller2015,miller2016,hua2019,pramanik2019,pramanik2020motivation,polansky2021,pramanik2021,pramanik2021e}. Furthermore, this approach can be used to obtain a solution for stability of an economy after pandemic crisis \citep{ahamed2021}, determine an optimal bank profitability \citep{islamrole,alamanalyzing,mohammad2010,alam2013,hossain2015,pramanik2016,alam2018,ahamed2021d,alam2021,alam2021output}. This approach can be used in population dynamics problems such as \cite{minar2018}, \cite{minar2019}, \cite{minar2021} and \cite{minarrevisiting}. In a very competitive tournament like a soccer World Cup, all possible standard strategies to score  goals are known to the opposition teams. In this environment if a player's action is stochastic in nature then he would have some comparative advantage which is also known to the opposition team but, they do not know what type stochastic action is going to take place and his complete stochastic action profile is unknown to the opposition team. Apart from that, the conditions like uncertainties due to rain, dribbling and passing skill of a player, type of match (i.e. day match or a day-night match), and advantages of having home crowd have been considered as stochastic component of the goal dynamics on the way to determine the optimal weight. 

Recent literature talks about whether a soccer team should choose offensive or defensive strategies \citep{santos2014}. Some studies say that, relatively new ``Three point" rule and ``Golden goal" do not necessarily create to brake a tie and score goals \citep{brocas2004, santos2014} and if asymmetry between two teams is big then these two rules induce the weaker team to play more defensively \citep{guedes2002}. There are some other studies  combined them give mixed results on these two rules \citep{dilger2009,garicano2005, moschini2010, santos2014}. Therefore, we do not include these two rules in our analysis. On the other hand, as scoring a goal on a given condition of a match is purely stochastic, a discounted reward on a player's dynamic objective function can give them more incentive to score which is a common dynamic reward phenomenon in the animal kingdom \citep{kappen2007}. 

\section{Construction of the problem}
In this section we construct a forward stochastic goal dynamics under a Liouville-like quantum gravity action space with a conditional expected dynamic objective function. The objective function gives an expected number of goals of a match based on total number of goals scored by a team at the beginning of each time interval. For example at the beginning of a match both of the teams start with $0$ goals. Therefore, the initial condition is $Z_0=0$, where $Z_0$ represents total number of goals scored by a team at time $0$ of the interval $[0,t]$. The objective of player $i\in I$ at the beginning of $(M+1)^{th}$ game is :
\begin{multline}\label{f0}
\mathbf {OB}_\a^i:\overline{\mathbf Z}_\a^i(\mathbf W,s)\\=h_0^{i*}+\max_{W_i\in W} \E_0\left\{\int_{0}^t\sum_{i=1}^I\sum_{m=1}^M\exp(-\rho_s^im)\a^iW_i(s)h_0^i[s,w(s),z(s)]\bigg|\mathcal F_0^Z\right\}ds,
\end{multline}
where $W_i$ is the strategy of player $i$ (control variable), $\a^i\in\mathbb R$ is constant weight, $I$ is total number of players in a team including those at the reserve bench, $\rho_s^i\in(0,1)$ is a stochastic discount rate for player $i$ with $w\in\mathbb R_+^{I'}$ and $z\in\mathbb R_+^{I'}$ are time $s\geq0$ dependent all possible controls and goals available to them, $h_0^{i*}\geq 0$ is the initial condition of the function $h_0^i$ and $\mathcal F_0^Z$ is the filtration process of goal dynamics starting at the beginning of the game. Therefore, for player $i$ the difference between  $W_i$ and $w$ is that strategy $W_i$ is the subset of all possible strategies $w$ available to them at time $s$ before $(M+1)^{th}$ game starts. Furthermore, we assume $h_0^i$ is a known objective function to player $i$, which is  partly unknown to the opposition team because of incomplete and imperfect information of them.

Suppose, the stochastic differential equation corresponding to goal dynamics is 
\begin{equation}\label{f1}
d\mathbf Z(s,\mathbf W)=\bm\mu[s,\mathbf W(s),\mathbf{Z}(s,\mathbf W)]ds+\bm\sigma[s,\hat{\bm{\sigma}},\mathbf W(s),\mathbf{Z}(s,\mathbf W)]d\mathbf B(s),
\end{equation}
where $\mathbf W_{I\times I'}(s)\subseteq\mathcal W\subset\mathbb R_+^{I\times I'}$ is the control space and $\mathbf Z_{I\times I'}(s)\subseteq\mathcal Z\subset\mathbb R_+^{I\times I'}$ is space of scoring goals under the soccer rules such that $z\in\mathbf Z$, $\mathbf B_{p\times 1}(s)$ is a $p$-dimensional Brownian motion, $\bm\mu_{I\times 1}>0$ is the drift coefficient and the positive semidefinite matrix $\bm\sigma_{I\times p}\geq 0$ is the diffusion coefficient such that 
\[
\lim_{s\downarrow \infty} \E\mu[s,\mathbf W(s),\mathbf{Z}(s,\mathbf W)]=\mathbf Z^*\geq0.
\]
 Above argument states that, if enough time is allowed for a match then, $\mathbf Z^*$ number of goals would be achieved which turns out to be a stable solution of this system. Finally,
 \begin{equation}\label{f2}
 \bm\sigma[s,\hat{\bm{\sigma}},\mathbf W(s),\mathbf{Z}(s,\mathbf W)]=\gamma\hat{\bm{\sigma}}+\bm\sigma^*[s,\mathbf W(s),\mathbf{Z}(s,\mathbf W)],
 \end{equation}
 where $\hat{\bm\sigma}>0$ comes from the strategies of the opposition team with the coefficient $\gamma>0$ and $\bm\sigma^*>0$ comes from the weather conditions, venues, popularity of a club or a team before starting of $(M+1)^{th}$ game. The forward stochastic differential Equation (\ref{f2}) is the core of our analysis. We use all possible important conditions during a game.
 
 \section{Definitions and Assumptions}
 
 \begin{as}\label{asf0}
 	For $t>0$, let ${\bm{\mu}}(s,\mathbf{W},\mathbf{Z}):[0,t]\times \mathbb{R}^{I\times I'}\times \mathbb{R}^{I\times I'} \ra\mathbb{R}^{I\times I'}$ and $\bm{\sigma}(s,\hat{\bm{\sigma}},\mathbf{W},\mathbf{Z}):[0,t]\times \mathbb{S}^{(I\times I')\times t}\times \mathbb{R}^{I\times I'}\times \mathbb{R}^{I\times I'} \ra\mathbb{R}^{I\times I'}$ be some measurable function with $(I\times I')\times t$-dimensional two-sphere $\mathbb{S}^{(I\times I')\times t}$ and, for some positive constant $K_1$, $\mathbf{W}\in\mathbb{R}^{I\times I'}$ and, $\mathbf{Z}\in\mathbb{R}^{I\times I'}$ we have linear growth as
 	\[
 	|{\bm{\mu}}(s,\mathbf{W},\mathbf{Z})|+
 	|\bm{\sigma}(s,\hat{\bm{\sigma}},\mathbf{W},\mathbf{Z})|\leq 
 	K_1(1+|\mathbf{Z}|),
 	\]
 	such that, there exists another positive, finite, constant $K_2$ and for a different score vector 
 	$\widetilde{\mathbf{Z}}_{(I\times I')\times 1}$ such that the Lipschitz condition,
 	\[
 	|{\bm{\mu}}(s,\mathbf{W},\mathbf{Z})-
 	{\bm{\mu}}(s,\mathbf{W},\widetilde{\mathbf{Z}})|+|\bm{\sigma}(s,\hat{\bm{\sigma}},
 	\mathbf{W},\mathbf{Z})-\bm{\sigma}(s,\hat{\bm{\sigma}},\mathbf{W},\widetilde{\mathbf{Z}})|
 	\leq K_2\ |\mathbf{Z}-\widetilde{\mathbf{Z}}|,\notag
 	\]
 	$ \widetilde{\mathbf{Z}}\in\mathbb{R}^{I\times I'}$ is satisfied and
 	\[
 	|{\bm{\mu}}(s,\mathbf{W},\mathbf{Z})|^2+
 	\|\bm{\sigma}(s,\hat{\bm{\sigma}},\mathbf{W},\mathbf{Z})\|^2\leq K_2^2
 	(1+|\widetilde{\mathbf{Z}}|^2),
 	\]
 	where 
 	$\|\bm{\sigma}(s,\hat{\bm{\sigma}},\mathbf{W},\mathbf{Z})\|^2=
 	\sum_{i=1}^I \sum_{j=1}^I|{\sigma^{ij}}(s,\hat{\bm{\sigma}},\mathbf{W},\mathbf{Z})|^2$.
 \end{as}
 
 \begin{as}\label{asf1}
 	There exists a probability space $(\Omega,\mathcal{F}_s^{\mathbf Z},\mathcal{P})$ with sample space $\Omega$, filtration at time $s$ of goal ${\mathbf{Z}}$ as $\{\mathcal{F}_s^{\mathbf{Z}}\}\subset\mathcal{F}_s$, a probability measure $\mathcal{P}$ and a $p$-dimensional $\{\mathcal{F}_s\}$ Brownian motion $\mathbf{B}$ where the measure of valuation of players $\mathbf{W}$ is an $\{\mathcal{F}_s^{\mathbf{Z}}\}$ adapted process such that Assumption \ref{asf0} holds, for the feedback control measure of players there exists a measurable function $h$ such that $h:[0,t]\times C([0,t]):\mathbb{R}^{I\times I'}\ra\mathbf{W}$ for which $\mathbf{W}(s)=h[\mathbf{Z}(s,w)]$ such that Equation (\ref{f1})	has a strong unique solution \citep{ross2008}.
 \end{as}

\begin{as}\label{asf3}
(i). $\mathcal Z\subset\mathbb R^{I\times I'}$ such that a soccer player $i$ cannot go beyond set $\mathcal Z_i\subset \mathcal Z$ because of their limitations of skills.  This immediately implies set $\mathcal Z_i$ is different for different players.\\
(ii). The function $h_0^i:[0,t]\times\mathbb R^{2I'}\ra\mathbb R^{I'}$. Therefore, all players in a team at the beginning of $(M+1)^{th}$ match have the objective function $h_0:[0,t]\times\mathbb R^{I\times I'}\times \mathbb R^{I\times I'}\ra\mathbb R^{I\times I'}$ such that $h_0^i\subset h_0$ in functional spaces and both of them are concave which is equivalent to Slater condition \citep{marcet2019}.\\
(iii). There exists an $\bm\epsilon>0$ such that for all $(\mathbf W,\mathbf Z)$ and $i=1,2,...,I$ such that 
\[
\E_0\left\{\int_{0}^t\sum_{i=1}^I\sum_{m=1}^M\exp(-\rho_s^im)\a^iW_i(s)h_0^i[s,w(s),z(s)]\bigg|\mathcal F_0^Z\right\}ds\geq\bm\epsilon.
\]
\end{as}

\begin{definition}\label{def0}
	Suppose $\mathbf{Z}(s,\mathbf W)$ is a non-homogeneous Fellerian semigroup on time in $\mathbb{R}^{I\times I'}$. The infinitesimal generator $A$ of $\mathbf{Z}(s,\mathbf W)$ is defined by,
	\[
	Ah(z)=\lim_{s\downarrow 0}\frac{\E_s[h(\mathbf{Z}(s,\mathbf W))]-h(\mathbf Z(\mathbf W))}{s},
	\]
	for $\mathbf Z\in\mathbb{R}_+^{I\times I'}$	where $h:\mathbb{R}_+^{I\times I'}\ra\mathbb{R}_+$ is a $C_0^2(\mathbb{R}_+^{I\times I'})$ function, $\mathbf{Z}$ has a compact support, and at $\mathbf Z(\mathbf W)>\bm 0$ the limit exists where $\E_s$ represents the soccer team's conditional expectation of scoring goals $\mathbf{Z}$ at time $s$. Furthermore, if the above Fellerian semigroup is homogeneous on times, then $Ah$ is the Laplace operator.
\end{definition}

\begin{defn}\label{def1}
	For a Fellerian semigroup $\mathbf{Z}(s,\mathbf W)$ for all $\epsilon>0$, the time interval $[s,s+\epsilon]$ with $\epsilon\downarrow 0$, define a characteristic-like quantum operator starting at time $s$ is defined as 
	\[
	\mathcal{A} h(\mathbf Z)=\lim_{\epsilon\downarrow 0}
	\frac{\log\E_s[\epsilon^2\ h(\mathbf{Z}(s,\mathbf W))]-\log[\epsilon^2h(\mathbf Z(\mathbf W))]}{\log\E_s(\epsilon^2)},
	\]
	for $\mathbf Z\in\mathbb{R}_+^{I\times I'}$,	where $h:\mathbb{R}^{I\times I'}\ra\mathbb{R}$ is a $C_0^2\left(\mathbb{R}_+^{I\times I'}\right)$ function, $\E_s$ represents the conditional expectation of goal dynamics $\mathbf{Z}$ at time $s$,  for $\epsilon>0$ and a fixed $h$ we have the sets of all open balls of the form $B_\epsilon(h)$ contained in $\mathcal{B}$ (set of all open balls) and as $\epsilon\downarrow 0$ then $\log\E_s(\epsilon^2)\ra\infty$.
\end{defn}

\begin{defn}\label{def2}
Following \cite{frick2019} a dynamic conditional expected objective function explained in Equation (\ref{f0}) on the goal dynamics $\mathbf Z\in\{\mathbf Z_0,\mathbf Z_1,...,\mathbf Z_t\}$ is a tuple $\left(s,\bm\a^i,\{\mathbf{OB}_\a^i(W_i),\bm\tau_{W_i}\}_{W_i\in w}\right)$ where\\
(i). $w$ is a finite strategy space where player $i$ can choose strategy $W_i$ and $\bm\a^i$ is all probabilities available to them from where they can choose $\a^i$.\\
(ii). For each strategy $W_i\in w$, $\mathbf{OB}_\a^i\in\mathbb R^{\mathbf Z}$ is constrained objective function of soccer player $i$ such that Definition \ref{def1} holds.\\
(iii). For each strategy $W_i\in w$, define the rain or other environmental random factors which leads to a stoppage or termination of game $M+1$ at time $s$ as $\bm\tau_{W_i}$, which is a finitely-additive probability measure on the Borel $\sigma$-algebra on $\mathbb R^{\mathbf Z}$ and is proper.
\end{defn}

 Following \cite{hellman2019}, two main types of logic are used in game theory: First-order Logic and Infinitary Logic. First-order mathematical logic is built on finite base language based on the connective symbols such as conjunction, negation, conversion, inversion and  contrapositivity; countable collection of variables, quantifier symbols, constant symbols, predicate symbols and function symbols \citep{hellman2019}. Quantum formulae are of the form of $h(0,...,t)$ such that the game operates in a quantum field with the characteristic-like quantum generator defined in Definition \ref{def1}. Comparing this statement with the interpretation of Atomic formulae in \cite{hellman2019} we can say $h$ is a functional predicate symbol on terms $0,...,t$. Our Quantum formulae are more generalized version of Atomic formulae in the sense that, Quantum formulae consider improper differentiability on  $2$-sphere continuous strategy available for all the players in both of the teams.
 
 The problem of constructing a First-order logic is that, it can only handle expressions  of finite variables. Dealing with infinite variables in strategy space is the primary objective in our paper. When player $i$ tries to score a goal, they change their actions based on the strategies of the opposition team and on their skills. If player $i$ is a senior player, then opposition team has more information about that player's strength or weakness which is also known to player $i$. Therefore, at time $s$ of match $M+1$ player $i$'s action is mixed. Furthermore, we assume each player's strategy set is a convex polygon with each side has the length of unity. The reason is that, probability of choosing a strategy is in between $0$ and $1$. Therefore, if a player has three strategies, their strategy set is a equilateral triangle with each side of length unity.
   
 To get more generalized result we extend the standard First-order logic to Infinitary logic \citep{hellman2019}. This logic considers mixed actions with infinite possible strategies such that for each player is able to play a combination of infinitely many strategies at infinite number of states. For example, a striker $i$ gets the ball at time $s$ either from their team mate or by a result of a missed pass from an opponent. As striker $i$'s objective is to kick the ball through the goal, their strategy depends on to total number of opponents between them and the goal. Therefore, striker $i$ plays a mixed action or striker $i$ places weight $\a^i$ on scoring strategy $a^i$ at total number of goal scored $\mathbf Z_s$ at time $s$. For a countable collection of objective functions $\{{\bf{OB}}_\a^i\}_{i=1}^\infty$ such that $\bigwedge_{i=1}^\infty{\bf{OB}}_\a^i$ and $\bigvee_{i=1}^\infty{\bf{OB}}_\a^i$ exist, converges to $\overline{\mathbf Z}_\a$ in real numbers via the formula 
 \[
 {\bf{OB}}_\a\left(\{\overline{\mathbf Z}_\a^i\}_{i=1}^I, \overline{\mathbf Z}_\a\right)=\forall\varepsilon\left[\varepsilon>0\ra\bigvee_{I\in\mathbb N}\bigwedge_{i>I}\left(\overline{\mathbf Z}_\a^i-\overline{\mathbf Z}_\a\right)^2<\varepsilon^2\right]
 \]
 or, without the quantifier the above statement becomes
 \[
 {\bf{OB}}_\a\left(\{\overline{\mathbf Z}_\a^i\}_{i=1}^I, \overline{\mathbf Z}_\a\right)=\bigwedge_{\mathcal K\in\mathbb N}\bigvee_{I\in\mathbb N}\bigwedge_{i>I}\mathcal K^2\left(\overline{\mathbf Z}_\a^i-\overline{\mathbf Z}_\a\right)^2<1.
 \]
 
 \subsection{$\sqrt{8/3}$ Liouville quantum gravity surface}
 Following \cite{gwynne2016} we know, a Liouville quantum gravity (LQG) surface is a random Riemann surface parameterized by a domain $\mathbb D\subset\mathbb S^{(I\times I')\times t}$ with Riemann metric tensor $e^{\gamma k(l)} d\mathbf{Z}\otimes d\widehat{\mathbf{Z}}$, where $\gamma\in(0,2)$, $k$ is some variant of the Gaussian free field (GFF) on $\mathbb D$, $l$ is some number coming from $2$-sphere $\mathbb S^{(I\times I')\times t}$ and $d\mathbf{Z}\otimes d\widehat{\mathbf{Z}}$ is Euclidean metric tensor. In this paper we consider the case where $\gamma=\sqrt{8/3}$ because, it corresponds to a uniformly random planer maps. Because of incomplete and imperfect information the strategy space is quantum in nature and each player's decision is a point on a dynamic convex strategy polygon of that quantum strategy space. Furthermore, $k:\mathbb S^{(I\times I')\times t}\ra \mathbb R^{I\times I'}$ is a distribution such that each player's action which can be represented different shots to the goal including dribbling and passing. Although the strategy set is deterministic, the action on this space is stochastic.
 
 \begin{defn}
 	An equivalence relation $\mathcal E$ on $\mathbb S$ is smooth if there is another $2$-sphere $\mathbb S'$ and a distribution function with a conformal map $k':\mathbb S\ra\mathbb S'$ such that for all $l_1,l_2\in\mathbb S$, we have $l_1\mathcal E l_2\iff k(l_1)=k(l_2)$.
 \end{defn}

Now if $\mathcal E$ is the equivalence relation of each player's action, the smooth distribution function $k'$ is an auxiliary tool which helps determining whether $l_1$ and $l_2$ are in the same action component which occurs iff $k(l_1)=k(l_2)$.

\begin{exm}
	Suppose $\mathbb S^I=\mathbb C^I$, where $\mathbb C^I$ represents a complex space. The relation given $l_1\sim_{\mathcal E} l_2$ if and only if $l_1-l_2\in\mathbb S'$ is smooth as the distribution $k:\mathbb C^I\ra[0,1)^I$ is defined by $k(l_1,...,l_I)=(l_1-\lfloor l_1\rfloor,...,l_I-\lfloor l_I\rfloor )$, where $\lfloor l_m\rfloor=\max\{l^*\in\mathcal S'|l^*>l_m\}$ is the integer part of $l_m$, then $k(l_1)=k(l_2)$ iff $l_1\mathcal E l_2$.
\end{exm} 

Furthermore, $\sqrt{8/3}$-LQG surface is an equivalence class of action on $2$-sphere $(D,k)$ such that $D\subset\mathbb S^{(I\times I')\times t}$ is open and $k$ is a distribution function which is some variant of a GFF \citep{gwynne2016}. Action pairs $(D,k)$ and $(\widetilde D,\tilde k)$ are equivalent if there exists a conformal map $\zeta:\widetilde D\ra D$ such that, $\tilde k=k\circ\zeta+Q\log|\zeta'|$, where $Q=2/\gamma+\gamma/2=\sqrt{3/2}+\sqrt{2/3}$ \citep{gwynne2016}.

Suppose, $I$ be a non-empty finite set of players, $\mathcal F_s^{\bf Z}$ be the filtration of goal $\bf Z$, $\Omega$ be a sample space. There for each player $i\in I$ an equivalence relationship $\mathcal E_i\in\mathcal E$ on $2$-sphere, called player $i$'s quantum knowledge. Therefore, $\sqrt{8/3}$-LQG player knowledge space at time $s$ is $(\Omega,\mathcal F_s^{\bf Z},\mathbb S, I,\mathcal E)$. Given a $\sqrt{8/3}$-LQG player knowledge space $(\Omega,\mathcal F_s^{\bf Z},\mathbb S, I,\mathcal E)$ the equivalence relationship $\mathcal E$ is the transitive closure of $\bigcup_{i\in I}\mathcal E_i$.

\begin{defn}
	A knowledge space $(\Omega,\mathcal F_s^{\bf Z},\mathbb S, I,\mathcal E)$ such that $i\in I$, each equivalent class of $\mathcal E_i$ with Riemann metric tensor $e^{\sqrt{8/3}k(l)}d\bf Z\otimes d\widehat{\mathbf{Z}}$ is finite, countably infinite or uncountable is defined as purely $\sqrt{8/3}$-LQG knowledge space which is purely quantum in nature (for detailed discussion about purely atomic knowledge see \cite{hellman2019}).
\end{defn}

\begin{defn}
	For a fixed quantum knowledge space $(\Omega,\mathcal F_s^{\bf Z},\mathbb S, I,\mathcal E)$, for player $i$ a dribbling and passing function $p^i$ is a mapping $p^i:\Omega\times\mathbb S\ra\Delta(\Omega\times \mathbb S)$ which is $\sigma$-measurable and the equivalence relationship has some measure in $2$-sphere.
\end{defn}
Therefore, dribbling and passing space is a tuple $(\Omega,\mathcal F_s^{\bf Z},\mathbb S, I,p)$ which is a type of $\sqrt{8/3}$-LQG knowledge space. There are other skills needed to score a goal such as power, speed, agility, shielding, tackling, trapping and shooting but we assume only dribbling and passing function is directly related to the quantum knowledge. Rest of the uncertainties are coming from the stochastic part of the goal dynamics. The dribbling and passing space of player $i$ implicitly defines quantum knowledge relationship $\mathcal E_i$ of dribbling and passing functions. Hence, $(z,l)\mathcal E_i(z',l')$ iff $p_{z,l}^i=p_{z',l'}^i$, where $z$ is the goal situation according to player $i$'s perspective and $l$ is player $i$'s action on $2$-sphere at time $s$.
\begin{defn}
	For player $i\in I$ and for all $z\in\Omega$, $l\in\mathbb S$ dribbling and passing space tuple $(\Omega,\mathcal F_s^{\bf Z},\mathbb S, I,p)$ has a function $p_{z,l}^i$ which is purely quantum. Therefore, this space is purely quantum dribbling and passing space.
\end{defn}

\begin{definition}
	For $(\Omega,\mathcal F_s^{\bf Z},\mathbb S, I,p)$ if $p_{z,l}^i[z,l]>0$ for all $i\in I$, $z\in\Omega$ and $l\in\mathbb S$ then it is positive. Furthermore, the dribbling and passing space is purely quantum and positive, then the knowledge space is $\sqrt{8/3}$-LQG.
\end{definition}

\begin{defn}
	A purely quantum space $(\Omega,\mathcal F_s^{\bf Z},\mathbb S, I,p)$ is smooth if $\Omega$, the $2$-sphere $\mathbb S$ and the common quantum knowledge equivalence relation $\mathcal E$ is smooth. As $k$ is a version of GFF the quantum space is not smooth at the vicinity of the singularity and for time being we exclude those point to make dribbling and the passing space smooth. 
\end{defn}
\begin{example}
	Suppose, a male soccer team has three strikers A, B and C such that player A is a left wing, B is a center-forward and C is a right wing. Furthermore, at time $s$ player B faces an opposition center defensive midfielder (CDM) with at least one of the center backs (left or right) and decides to pass the ball either of players A and C as they have relatively unmarked positions. Player A knows, if B passes him, based on goal condition $z_1$ based on his judgment he will take an attempt to score a goal either by taking a long distance shot with conditional probability $u(u_1|z_1)$ or by running with the ball closer to the goal post with probability $u(u_2|z_1)$, where for $k=1,2$, $u_k$ takes the value $1$ when player A is able to score by using any of the two approaches. As the dribbling and passing space is purely quantum, player A's expected payoff to score is $u(u_1|z_1)e^{\sqrt{8/3}k_1(l_1)}+u(u_2|z_1)e^{\sqrt{8/3}k_1(l_2)}$, where $l_k$ is the $k^{th}$ point observed on player A's $2$-sphere $\mathbb S$. Similarly, player C has the payoff of scoring a goal is $u(u_3|z_2)e^{\sqrt{8/3}k_2(l_3)}+u(u_4|z_2)e^{\sqrt{8/3}k_2(l_4)}$, where $u_3$ and $u_4$ are the probabilities of taking a direct shot to goal and running closer to goal post and score. However, player B's decision tends to be somewhat scattered, as he has to make a choice of passing either of A or C or mixes up the calculation and decides to goal by himself. Consider player B mixes up strategy and decides to goal by himself with probability $v$ or he passes either of A and C with probability $(1-v)$, which is know to players A and C.
	
	Let us model the quantum knowledge space as of these three players as $\Omega=\mathbb{R}_+\times \mathbb{R}_+\times \mathbb{R}_+$ and $\mathbb S=\mathbb{R}_+\times \mathbb{R}_+\times \mathbb{R}_+$. Once, the payoff left wing A and right wing C has been revealed, player B calculates the ratio of of two expected payoffs. Therefore, player A knows that, player B is is giving him a pass if 
	\[
	\frac{u(u_1|z_1)e^{\sqrt{8/3}k_1(l_1)}+u(u_2|z_1)e^{\sqrt{8/3}k_1(l_2)}}{u(u_3|z_2)e^{\sqrt{8/3}k_2(l_3)}+u(u_4|z_2)e^{\sqrt{8/3}k_2(l_4)}}>1
	\]
	with probability $(1-v)$ and with probability $v$ player B decides to score himself if the ratio is less than equal to $1$. As both of the wing players calculate the ratio of their expected payoffs, it is enough to define $\Omega=\mathbb R_+\times \mathbb R_+$ and $\mathbb S=\mathbb R_+\times \mathbb R_+$ for two wing players without losing any information. The quantum equivalent classes $\mathcal E^A$ which is player A's knowledge can be represented as
	\begin{align*}
	&\left[\left(0,\frac{u(u_1|z_1)e^{\sqrt{8/3}k_1(l_1)}+u(u_2|z_1)e^{\sqrt{8/3}k_1(l_2)}}{u(u_3|z_2)e^{\sqrt{8/3}k_2(l_3)}+u(u_4|z_2)e^{\sqrt{8/3}k_2(l_4)}}>1\right),\right.\\&\left.\left(0,\frac{u(u_1|z_1)e^{\sqrt{8/3}k_1(l_1)}+u(u_2|z_1)e^{\sqrt{8/3}k_1(l_2)}}{u(u_3|z_2)e^{\sqrt{8/3}k_2(l_3)}+u(u_4|z_2)e^{\sqrt{8/3}k_2(l_4)}}\leq 1\right)\right]_{(z_k\in\mathbb R_+,l_k\in \mathbb R_+,\ \text{for all $k=1,2$})},
	\end{align*}
	and player B's knowledge corresponding to the equivalent class $\mathcal E^B$ is 
	\begin{align*}
&\left[\left(\frac{u(u_3|z_2)e^{\sqrt{8/3}k_2(l_3)}+u(u_4|z_2)e^{\sqrt{8/3}k_2(l_4)}}{u(u_1|z_1)e^{\sqrt{8/3}k_1(l_1)}+u(u_2|z_1)e^{\sqrt{8/3}k_1(l_2)}}>1,0\right),\right.
\\ &\left.\left(\frac{u(u_3|z_2)e^{\sqrt{8/3}k_2(l_3)}+u(u_4|z_2)e^{\sqrt{8/3}k_2(l_4)}}{u(u_1|z_1)e^{\sqrt{8/3}k_1(l_1)}+u(u_2|z_1)e^{\sqrt{8/3}k_1(l_2)}}\leq 1,0\right)\right]_{(z_k\in\mathbb R_+,l_k\in \mathbb R_+,\ \text{for all $k=1,2$})}.
	\end{align*}
	Therefore, the belief of player A is 
	\begin{align*}
	p_{z,l}^A[z,l]&=p_{[u(u_1|z_1)e^{\sqrt{8/3}}k_1(l_1),u(u_2|z_1)e^{\sqrt{8/3}}k_1(l_2)]}^A[z,l]\\&=\begin{cases}
	1-v & \text{if}\  \frac{u(u_1|z_1)e^{\sqrt{8/3}k_1(l_1)}+u(u_2|z_1)e^{\sqrt{8/3}k_1(l_2)}}{u(u_3|z_2)e^{\sqrt{8/3}k_2(l_3)}+u(u_4|z_2)e^{\sqrt{8/3}k_2(l_4)}}>1,\\
	v & \text{if}\ \frac{u(u_1|z_1)e^{\sqrt{8/3}k_1(l_1)}+u(u_2|z_1)e^{\sqrt{8/3}k_1(l_2)}}{u(u_3|z_2)e^{\sqrt{8/3}k_2(l_3)}+u(u_4|z_2)e^{\sqrt{8/3}k_2(l_4)}}\leq 1,\\
	0 & \text{otherwise},
	\end{cases}
	\end{align*}
	and similarly for the right wing B,
\begin{align*}
p_{z,l}^B[z,l]&=p_{u(u_3|z_2)e^{\sqrt{8/3}k_2(l_3)},u(u_4|z_2)e^{\sqrt{8/3}k_2(l_4)}}^B[z,l]\\&=\begin{cases}
1-v & \text{if}\ \frac{u(u_3|z_2)e^{\sqrt{8/3}k_2(l_3)}+u(u_4|z_2)e^{\sqrt{8/3}k_2(l_4)}}{u(u_1|z_1)e^{\sqrt{8/3}k_1(l_1)}+u(u_2|z_1)e^{\sqrt{8/3}k_1(l_2)}}>1,\\
v & \text{if}\ \frac{u(u_3|z_2)e^{\sqrt{8/3}k_2(l_3)}+u(u_4|z_2)e^{\sqrt{8/3}k_2(l_4)}}{u(u_1|z_1)e^{\sqrt{8/3}k_1(l_1)}+u(u_2|z_1)e^{\sqrt{8/3}k_1(l_2)}}\leq 1,\\
0 &\text{otherwise}.
\end{cases}	
\end{align*}
\end{example}

\subsection{$\sqrt{8/3}$ Liouville quantum gravity metric}
For domain $\mathbb D\subset \mathbb S^{(I\times I')\times t}$ and for a variant of GFF the distribution $k$ suppose, $(\mathbb D, k)$ is a $\sqrt{8/3}$-LQG action surface. As $k$ is a variant of GFF, $k$ induces a metric $\bm \omega_k$ on $\mathbb{D}$ and furthermore, if $(\mathbb D, k)$ is a quantum sphere, the metric space $(\mathbb D, \bm\omega_k)$ is isometric to the Brownian map with two additional properties: a Brownian disk is obtained if $(\mathbb D, k)$ is a quantum disk and a Brownian plane is obtained if $(\mathbb D, k)$ is a $\sqrt{8/3}$-quantum cone \citep{bettinelli2017,curien2014,gwynne2016}. Following \cite{gwynne2016}  we know that, $\sqrt{8/3}$-LQG surface can be represented as a Brownian surface with conformal structure. We assume the strategy space where the action is taken has the property like $\sqrt{8/3}$-LQG surface because, each soccer player has radius $r$ around themselves such that, if an opponent player comes in this radius he would be able to tackle. Furthermore, if $r\ra 0$, the player has a complete control over the opponent under no mistakes and same skill level. Therefore, the strategy space closer to the player (i.e. $r=0$) bends towards himself in such a way that, the surface can be approximated to a surface on a $2$-sphere and furthermore, as the movement on this space is stochastic in nature, it behaves like a Brownian surface with its convex strategy polygon changes its shape at every time point based on the condition of the game. At $r=0$ the surface hits essential singularity and the player has infinite power to control over the ball. 

The $\sqrt{8/3}$-LQG metric will be constructed on a $2$-sphere $(\mathbb S,k)$ which is also quantum in nature. Based on the distribution function $k$ suppose, $C_i$ be the collection of i.i.d. locations of player $i$ on the strategy space sampled uniformly from the area of their convex dynamic polygon $\be_{k_i}$. Furthermore, consider inside a polygon with area measure $\be_{k_i}$, player $i$'s action at time $s$ is $a_{s}^i$ and at time $\tau$ is $a_\tau^i$ such that, for $\varepsilon>0$ we define $\tau=s+\varepsilon$. Therefore, $a_s^i,a_\tau^i\in C_i$ is a Quantum Loewner Evolution growth process $\{\mathcal G_{\tilde s}^{a_s^i,a_\tau^i}\}_{\tilde s\geq 0}$ starting from $a_s^i$ and ending at $a_\tau^i$ for all $\tilde s\in[s,\tau]$. 

Let there are two growth processes $\{\mathcal G_{\tilde s}^{a_s^i,a_\tau^i}\}_{\tilde s\geq 0}$ and $\{\mathcal G_{*\tilde s}^{a_s^i,a_\tau^i}\}_{\tilde s\geq 0}$ where both of them starts at $a_s^i$ and ends at $a_\tau^i$. We will show on this surface they are homotopic in nature. Suppose $\{p_i^0,p_i^1,...,p_i^M\}$ are $M$-partitions on player $i$'s finite convex strategy polygon $\be_{k_i}$ on $\mathbb S^{(I\times I')\times t}$. Hence, each of $\be_k\in\be_{k_i}$ can be uniquely represented by $\be_k=\sum_{n=0}^M\theta_n(\be_k)p_i^n$, where $\theta_n(\be_k)\in[0,1]$ for all $n=0,...,M$ and $\sum_{n=0}^M\theta_n(\be_k)=1$. Thus,
\[
\theta_n(\be_k)=\begin{cases}
\text{barycentric coordinate of $\be_k$ relative to $p_i^n$ if $p_i^n$ is a vertex of $\be_{k_i}$,}\\ 0 \ \ \text{otherwise.}
\end{cases}
\]
Therefore, $\theta_n(\be_k)\neq 0$ are barycentric coordinates of $\be_k$. Furthermore, as $\theta_n$ is identically equal to zero or the barycentric function of $\be_{k_i}$ relative to the vertex $p_i^n$, $\theta_n:\be_{k_i}\ra \mathbf I$ is continuous, where $\mathbf I$ is the range of $p_i^n$ which takes the value between $0$ and $1$. Now assume the growth process has the map $\{\mathcal G_{\tilde s}^{a_s^i,a_\tau^i}\}_{\tilde s\geq 0}:\mathbf Y_i\ra\be_{k_i}$, where $\mathbf Y_i$ is an arbitrary topological space and it has a unique expression $\{\mathcal G_{\tilde s}^{a_s^i,a_\tau^i}\}_{\tilde s\geq 0}(\mathbf Y_i)=\sum_{n=0}^M\theta_n\circ p_i^n\{\mathcal G_{\tilde s}^{a_s^i,a_\tau^i}\}_{\tilde s\geq 0}$. As this growth process is a Quantum Loewner Evolution, it is continuous. Therefore, $\theta_n\circ \{\mathcal G_{\tilde s}^{a_s^i,a_\tau^i}\}_{\tilde s\geq 0}:\mathbf Y_i\ra \mathbf I$ is continuous. 

\begin{lem}
Let $\be_{k_i}$ be a finite convex strategy polygon of player $i$ in $\mathbb S^{(I\times I')\times t}$ such that the vertices are $\{p_i^0,...,p_i^M\}$. Suppose $\mathbf Y_i$ be an arbitrary topological space and let $\{\mathcal G_{n,\tilde s}^{a_s^i,a_\tau^i}\}_{\tilde s\geq 0}:\mathbf Y_i\ra\mathbf I$ be a family of growth processes, one for each vertex $p_i^n$, with $\sum_{n=0}^Mp_i^n\{\mathcal G_{n,\tilde s}^{a_s^i,a_\tau^i}(y_i)\}_{\tilde s\geq 0}\in\be_{k_i}$ for each $y_i\in\mathbf Y_i$. Then $y_i\mapsto\sum_{n=0}^Mp_i^n\{\mathcal G_{n,\tilde s}^{a_s^i,a_\tau^i}(y_i)\}_{\tilde s\geq 0}\in\be_{k_i}$ is a continuous map $\{\mathcal G_{\tilde s}^{a_s^i,a_\tau^i}\}_{\tilde s\geq 0}:\mathbf Y_i\ra\be_{k_i}$ as $\{\mathcal G_{\tilde s}^{a_s^i,a_\tau^i}\}_{\tilde s\geq 0}$ is a Quantum Loewner Evolution.
\end{lem}

\begin{proof}
	As $\{\mathcal G_{\tilde s}^{a_s^i,a_\tau^i}\}_{\tilde s\geq 0}$ is a Quantum Loewner Evolution, it is continuous. Furthermore, as $\{\mathcal G_{\tilde s}^{a_s^i,a_\tau^i}\}_{\tilde s\geq 0}(\mathbf Y_i)\subset\be_{k_i}\subset\mathbb S^{(I\times I')\times t}$, the continuity of $\{\mathcal G_{\tilde s}^{a_s^i,a_\tau^i}\}_{\tilde s\geq 0}$ as a map of $\mathbf Y_i$ into $\be_{k_i}$ follows.
\end{proof}

\begin{lem}
Let $\be_{k_i}$ be a finite convex strategy polygon of player $i$ in $\mathbb S^{(I\times I')\times t}$. Suppose, $\mathbf Y_i$ be an arbitrary space such that, $\{\mathcal G_{\tilde s}^{a_s^i,a_\tau^i}\}_{\tilde s\geq 0},\{\mathcal G_{*\tilde s}^{a_s^i,a_\tau^i}\}_{\tilde s\geq 0}:\mathbf Y_i\ra\be_{k_i} $ is a Quantum Loewner Evolution. For each $y_i\in\mathbf Y_i$ suppose there is a smaller finite convex strategy polygon $\rho\in\be_{k_i}$ containing both $\{\mathcal G_{\tilde s}^{a_s^i,a_\tau^i}(y_i)\}_{\tilde s\geq 0}$ and $\{\mathcal G_{*\tilde s}^{a_s^i,a_\tau^i}(y_i)\}_{\tilde s\geq 0}$. Then these two growth processes $\{\mathcal G_{\tilde s}^{a_s^i,a_\tau^i}\}_{\tilde s\geq 0}$ and $\{\mathcal G_{*\tilde s}^{a_s^i,a_\tau^i}\}_{\tilde s\geq 0}$ are homotopic.
\end{lem}

\begin{proof}
	For each $y_i\in\mathbf Y_i$ the line joining two growth process $\{\mathcal G_{\tilde s}^{a_s^i,a_\tau^i}(y_i)\}_{\tilde s\geq 0}$ and $\{\mathcal G_{*\tilde s}^{a_s^i,a_\tau^i}(y_i)\}_{\tilde s\geq 0}$ is in $\rho$. Therefore, this line segment is definitely inside $\be_{k_i}$. The representation of this segment is
	\begin{align*}
	\mathcal S(y_i,U)=\sum_{n=0}^M \left[U\theta_n\left(\{\mathcal G_{\tilde s}^{a_s^i,a_\tau^i}(y_i)\}_{\tilde s\geq 0}\right)+(1-U)\theta_n\left(\{\mathcal G_{*\tilde s}^{a_s^i,a_\tau^i}(y_i)\}_{\tilde s\geq 0}\right)\right].
	\end{align*}
	Furthermore, as the mapping $U\theta_n\circ \{\mathcal G_{\tilde s}^{a_s^i,a_\tau^i}\}_{\tilde s\geq 0}+(1-U)\theta_n\circ \{\mathcal G_{*\tilde s}^{a_s^i,a_\tau^i}\}_{\tilde s\geq 0}:\mathbf Y_i\times\mathbf I\ra\mathbf I$ is continuous, the function $\mathcal S:\mathbf Y_i\times\mathbf I\ra\mathbf I$ is continuous. This is enough to show homotopy.
\end{proof}

\begin{cor}
Let $\be_{k_i}$ be a finite convex strategy polygon of player $i$ in $\mathbb S^{(I\times I')\times t}$. Suppose, $\{\mathcal G_{\tilde s}^{a_s^i,a_\tau^i}\}_{\tilde s\geq 0},\{\mathcal G_{*\tilde s}^{a_s^i,a_\tau^i}\}_{\tilde s\geq 0}:\be_{k_i}\ra\be_{k_i} $ is a Quantum Loewner Evolution. For each $\be_k\in\be_{k_i}$ suppose there is a smaller finite convex strategy polygon $\rho\in\be_{k_i}$ containing both $\{\mathcal G_{\tilde s}^{a_s^i,a_\tau^i}(\be_k)\}_{\tilde s\geq 0}$ and $\{\mathcal G_{*\tilde s}^{a_s^i,a_\tau^i}(\be_k)\}_{\tilde s\geq 0}$. Then these two growth processes $\{\mathcal G_{\tilde s}^{a_s^i,a_\tau^i}\}_{\tilde s\geq 0}$ and $\{\mathcal G_{*\tilde s}^{a_s^i,a_\tau^i}\}_{\tilde s\geq 0}$ are homotopic.	
\end{cor}

From above Corollary it is clear that, any growth process in the interval $[a_s^i,a_\tau^i]$ inside the finite strategy polygon $\be_{k_i}$ is homotopic. Therefore, without any other further restriction a soccer player can follow any path without hampering their payoff.

\begin{defn}
A collection $k=\{k_{(I\times I')\times t}\}$ which are assumed to be homomorphisms $k_{(I\times I')\times t}:\mathbb C^{(I\times I')\times t}(\be_{k_i})\ra\mathbb C^{(I\times I')\times t}(\hat\be_{k_i})$, one for each dimension $(I\times I')\times t\geq 0$, and such that $\partial_{[(I+1)\times(I'+1)]\times (t+1)}\circ k_{[(I+1)\times(I'+1)]\times (t+1)}=k_{(I\times I')\times t}\circ\ \partial_{[(I+1)\times(I'+1)]\times (t+1)}$ is called a chain transformation or chain map on $\mathbb S$ where, $\hat \be_{k_i}$ is a convex strategy polygon of player $i$ other than $\be_{k_i}$.
\end{defn}

For a finite convex polygon $\be_{k_i}$ we take the chains over $2$-sphere $\mathbb S$ and the chain transformation $k$ of $\mathbb C_*(\be_{k_i};\mathbb S)$ into itself. Then the following relationship of trace is going to hold.

\begin{lem}
(Hopt trace theorem, \cite{granas2003}) Suppose, the dimension of player $i$'s strategy polygon $\be_{k_i}$ has the dimension of $(I\times I')\times t$, and the collection of distribution function which follows
\[
k_*:\mathbb C_*(\be_{k_i},\mathbb S)\ra \mathbb C_*(\be_{k_i},\mathbb S),
\]
by any chain transformation. Then
\[
\sum_{m=0}^{(I\times I')\times t}(-1)^m\text{tr}(k_m)=\sum_{m=0}^{(I\times I')\times t}(-1)^m\text{tr}(k_{m*}),
\]
where $\text{tr}(.)$ represents trace of the argument.
\end{lem}

For a finite convex strategy polygon $\be_{k_i}$ in $\sqrt{8/3}$-LQG consider the map $\hat k:\be_{k_i}\ra\be_{k_i}$.	Using rational points on the surface $\mathbb Q$ as coefficients, then each induced homomorphism $\hat k_*^{(I\times I')\times t}:\mathbb H_{(I\times I')\times t}(\be_{k_i};\mathbb Q)\ra \mathbb H_{(I\times I')\times t}(\be_{k_i};\mathbb Q)$ is an endomorphism, where $\hat k_*$ is the linear transformation of the vector space. Following \cite{granas2003} each $\hat k_*^{(I\times I')\times t}$ has a trace $\text{tr}\left[\hat k_*^{(I\times I')\times t}\right]$.
\begin{defn}
For the strategy polygon of player $i$ with $\dim({\be_{k_i}})\leq(I\times I')\times t$ and the endomorphic map $\hat k:\be_{k_i}\ra\be_{k_i}$, the Lefschetz number $\o(\hat k)$ is 
\[
\o(\hat k)=\sum_{m=0}^{(I\times I')\times t}(-1)^m\text{tr}[\hat k_{m*},\mathbb H_m(\be_{k_i};\mathbb Q)].
\]
\end{defn}

\begin{lem}
\citep{granas2003} For a strategy polygon $\be_{k_i}$ if $\hat k:\be_{k_i}\ra\be_{k_i}$ is continuous then, $\o(\hat k)$ depends on the homotopy class of $\hat k$. Furthermore, the Lefschetz number $\o(\hat k)$ is an integer irrespective of any $\sqrt{8/3}$-LQG field characteristics.
\end{lem}

\begin{prop}\label{fp0}
(Lefschetz-Hopf fixed point theorem on $\sqrt{8/3}$-LQG surface) If player $i$ has a finite strategy polygon $\be_{k_i}$ on $\mathbb S$ with the map $\hat k:\be_{k_i}\ra\be_{k_i}$ then, $\hat k$ has a fixed point for all Lefschetz number $\o(\hat k)\neq0$.
\end{prop}

\begin{proof}
We will prove this theorem by contradiction. Let $\hat k$ has no fixed points. As player $i$'s strategy polygon $\be_{k_i}$ compact, for all $\be_k\in\be_{k_i}$ there $\exists\varepsilon>0$ such that the measure $d[\hat k(\be_k),\be_k]\geq\varepsilon$. In this proof a repeated barycentric subdivision of $\be_{k_i}$ will be used with a fixed quadragulation of mesh $<\varepsilon/n_i$, where $n_i\in\mathbb N$.

Let $\be_{k_i}^{(m)}$ is $m^{th}$ barrycentric subdivision of $\be_{k_i}$ such that the mapping $\psi:\be_{k_i}^{(m)}\ra\be_{k_i}$ be a simplicical approximation of $\hat k$. Define $m^{th}$ barrycentric subdivision's map when chain characteristic is present in $\text{sub}^m:\mathbb C_*(\be_{k_i})\ra\mathbb C_*(\be_{k_i}^{(m)})$. It is enough to determine the trace of $\psi\text{sub}^m:\mathbb C_q(\be_{k_i};\mathbb Q)\ra\mathbb C_q(\be_{k_i};\mathbb Q)$ for each oriented $q$-sided smaller finite strategy polygons inside $\be_{k_i}$, where $\mathbb C_q(\be_{k_i};\mathbb Q)$ is the chain characteristic on $q$-sided smaller strategy polygon. To determine trace we would use the Hopf trace theorem explained above.  Suppose, for each $\mathbb C_q(\be_{k_i};\mathbb Q)$ there exists a basis $\{\rho_l^q\}$ of all oriented $q$-sided smaller strategy polygons in $\be_{k_i}$. Expressing $\psi\text{sub}^m$ in terms of the basis function would be,
\[
\text{sub}^m\rho_l^q=\sum\a_{ll'}k_{l'}^q,\ \a_{ll'}=0,\pm 1,\ k_{l'}^q\in\be_{k_i}^{(m)}, \ k_{l'}^q\subset\rho_l^q.
\]
Hence,
\[
\psi\text{sub}^m\rho_l^q=\sum\a_{ll'}\psi(k_{l'}^q)=\sum \gamma_{ll'}\rho_{l'}^q,
\]
where $\gamma_{ll'}$ is another basis. Now suppose, $\nu$ is the vertex of any $k_{l'}^q\subset\rho_l^q$. Define $\nu_*:=\be_{k_i}\setminus\bigcup\{\rho\in\be_{k_i}|\nu\notin\rho\}$. Then for vertex $\nu$ we have $\hat k(\nu)\in\hat k(\nu_*)\subset\psi_*(\nu)$ such that $d[\hat k(\nu),\psi(\nu)]<\varepsilon/n_i$, where $\psi_*:=\be_{k_i}\setminus\{\rho\in\be_{k_i}|\psi\notin\rho\}$. Clearly, 
\[
d[\nu,\psi(\nu)]\geq d[\nu,\hat k(\nu)]-d[\hat k(\nu),\psi(\nu)]\geq 2\varepsilon/n_i,
\]
such that if $\psi(\nu)$ belongs to $\rho_l^q$ then, $\delta(\rho_l^q)\geq 2\varepsilon/n_i$, which is the contradiction of our argument that quadragulation is $<\varepsilon/n_i$. Therefore, $\text{tr}[\psi\text{sub}^m,\mathbb C_q(\be_{k_i};\mathbb Q)]=0$ for each $q$-sided finite smaller strategy space. Finally, by Hopf trace theorem we can say that Lefschetz number $\o(\hat k)=0$. This completes the proof.
\end{proof}	

According to \cite{gwynne2016} it is a continuum analog of the first passage percolation on a random planner map. Suppose, for $\tau>0$ and let $\eta_s^\tau$ be a whole plane of a Schramm-Loewner Evolution with the central charge $6$ ($SLE_6$) from $a_s^i$ to $a_\tau^i$ which is sampled independent with respect to $k$. Now the plane $\eta_s^\tau$ has been run by terminal time interval $\tau$ such that $\forall\ \varepsilon>0$ and determine it by $k$. For $\varepsilon>0$ and $\tilde s\in[s,\tau] $, suppose $\mathcal G_{\tilde s}^{a_s^i,a_\tau^i,\tau}:=\eta_s^\tau([s,\tau\wedge h_s^\tau])$, where $h_s^\tau$ is the first time $\eta_s^\tau$ hits $a_\tau^i$ when the player $i$ starts the process at $a_s^i$ and going towards $a_\tau^i$ \citep{gwynne2016}. Following \cite{miller2015} we know that, for $\varepsilon\downarrow 0$ a growing family of sets $\{\mathcal G_{\tilde s}^{a_s^i,a_\tau^i}\}_{\tilde s\geq 0}$ in the action interval $[a_s^i,a_\tau^i]$ can be found by taking almost sure limits of an appropriate chosen subsequence, which \cite{gwynne2016} calls as Quantum Loewner Evolution defined on $\sqrt{8/3}$-quantum sphere (QLE($8/3,0$)). For $\varepsilon\geq0$, suppose $\mathcal M_{\tilde s}^{a_s^i,a_\tau^i}(\bf Z, \bf W)$ be some length of the boundary of the connected set $\mathbb S\setminus\mathcal G_{\tilde s}^{a_s^i,a_\tau^i}$, such that it contains the terminal action $a_\tau^i$, where $\bf Z$ be all possible the goals, $\bf W$ be all possible control strategies available to player $i$ at time $s$. Furthermore, assume the stopping time due to reach action $a_\tau^i$ is $\sigma_{\mathcal M}^{a_s^i,a_\tau^i}>0$ define a measure $\mathcal M\geq 0$ such that 
\[
\mathcal M({\bf Z},{\bf W})=\int_s^{\sigma_{\mathcal M}^{a_s^i,a_\tau^i}}\frac{1}{\mathcal M_{\tilde s}^{a_s^i,a_\tau^i}({\bf Z}, {\bf W})}d\tilde s.
\]
 Define $\widehat{\mathcal G_{\mathcal M}^{a_s^i,a_\tau^i}}:=\mathcal G_{\sigma_{\mathcal M}^{a_s^i,a_\tau^i}}^{a_s^i,a_\tau^i}$. Then following \cite{gwynne2016} $\sqrt{8/3}$-LQG distance of $[a_s^i,a_\tau^i]$ is defined as 
 \[
 {\bm\omega}_k(a_s^i,a_\tau^i):=\inf\left\{\mathcal M({\bf Z},{\bf W})\geq 0;\ a_\tau^i\in\widehat{\mathcal G_{\mathcal M}^{a_s^i,a_\tau^i}} \right\}.
 \]
 Finally, by \cite{gwynne2016} and \cite{miller2016}, for all $\hat r_i\in\mathbb R$ the construction of this metric for $k+\hat r_i$ yields a scaling property such as 
 \[
 \bm{\omega}_{k+\hat r_i}(a_s^i,a_\tau^i)=e^{\left(\sqrt{8/3}\right)\hat r_i/4}\bm\omega_k(a_s^i,a_\tau^i).
 \]
 
 \subsection{Interpretation of the diffusion part of goal dynamics}
 
 Equation (\ref{f2}) talks about two components of the diffusion part, $\hat{\bm\sigma}$ which comes from the strategies from the opposition team and $\bm{\sigma}^*$ consists of the venue, the percentage of attendance of the home crowd, the type of match (i.e. day or day-night match), the amount of dew on the field and the speed of wind (which gives advantage to those free kickers who like make swings due to wind). 
 
 Firstly, consider the situation where a team is playing abroad. In this case the players have harder time scoring a goal than in their home environment. For example if an Argentine or a Brazilian player plays in an European club tournament it is extremely hard to score in European environment as the playing style is very different than their native land of Latin America. This thing is also true for World Cup where only Brazilian male soccer team is the only team from Latin America is able to win the cup in Europe ($1958$ FIFA World Cup in Sweden). As playing abroad would create extra mental pressure on the players, assume that pressure is a non-negative $C^2$ function $\mathbf{p}(s,\mathbf{Z}):[0,t]\times\mathbb{R}^{(I\times I')\times t}\ra\mathbb{R}_+^{(I\times I')\times t}$ at match $M+1$ such that if $\mathbf{Z}_{s-1}<\E_{s-1}(\mathbf{Z})$ then $\mathbf{p}$ takes a very high positive value. In other words, if actual number of goals at time $s-1$ (i.e., $\mathbf{Z}_{s-1}$) is less than the expected number of goals at that time then, the pressure to score a goal at $s$ is very high.
 
 Secondly, percentage of home crowd in the total crowd of the stadium matters in the sense that, if it is a home match for a team, then players get extra support from their fans and are motivated to score more. If a team has a mega-star, they will access more crowd. Generally, mega stars have fans all over the world, and therefore they might get more crowd from the opposition team than their team mates. Define a positive finite $C^2$ function $\mathbf{A}(u,\mathbf{W}):[0,t]\times\mathbb{R}^{(I\times I')\times t}\ra\mathbb{R}_+^{(I\times I')\times t}$ with $\partial \mathbf{A}/\partial\mathbf{W}>0$ and $\partial \mathbf{A}/\partial s\gtreqless 0$ depends on if at time $s$, player $i$ with valuation $W_i\in\mathbf{W}$ is still playing, or is out of the field due to injury or other reasons.
 
 Thirdly, if the match is a day match, then both teams have comparative advantage in better visibility due to the sun. On the other hand, if the match is a day-night match then a team's objective is to choose the side of the field in such a way that they get a better visibility and get advantage by scoring more goals. Therefore, in this game a team who lose the toss clearly has disadvantage in terms of the position of the field. Furthermore, because of the dew the ball becomes wet and heavier at night and, it would be harder to grip from a goal-keeper's point of view as well as move the ball and score from a striker's point of view. Therefore, if the toss winning team scores some goals in the first half, that team surely has some comparative advantage to win the game. Hence, winning the toss is important. Furthermore, if a team loses the toss, then its decision to score in the either of two halves depends on its opposition. Hence, define a function $\mathfrak{B}(\mathbf{Z})\in\mathbb{R}_+^{(I\times I')\times t}$ such that,
 \begin{align}\label{f3}
 \mathfrak{B}(\mathbf{Z})=\mbox{$\frac{1}{2}$} 
 [\mbox{$\frac{1}{2}$} \E_0(\mathbf{Z}_D^2)+ \mbox{$\frac{1}{2}$}\ \E_0(\mathbf{Z}_{DN}^1)]+
 \mbox{$\frac{1}{2}$} 
 [\mbox{$\frac{1}{2}$} \E_0(\mathbf{Z}_D^1)+\mbox{$\frac{1}{2}$} \E_0(\mathbf{Z}_{DN}^2)],
 \end{align}
 where for $i=1,2$, $\E_0(\mathbf{Z}_D^i)$ is the conditional expectation of goal of a team before the starting of the day match $M+1$ with total number of goals at $i^{th}$ half $\mathbf{Z}_D^i$, and $\E_0(\mathbf{Z}_{DN}^i)$ is the conditional expectation of the goal before starting a day-night match $M+1$. Furthermore, if a team wins the toss, then it will go for the payoff $\mbox{$\frac{1}{2}$} \left[\mbox{$\frac{1}{2}$} \E_0(\mathbf{Z}_D^2)+ \mbox{$\frac{1}{2}$} \E_0(\mathbf{Z}_{DN}^1)\right]$, and the later part of the Equation (\ref{f3}) otherwise.
 
 Finally, we consider the dew point measure and the speed of wind at time $s$ as an important factor in scoring a goal. As these two are natural phenomena and represent ergodic behavior, we assume this can be represented by a Weierstrass function $\mathbf{Z}_e:[0,t]\ra\mathbb{R}$ \citep{falconer2004} defined as,
 \begin{equation}\label{f4}
 \mathbf{Z}_e(s)=\sum_{\a=1}^{\infty}(\lambda_1+\lambda_2)^{(s-2)\a} \sin\left[(\lambda_1+\lambda_2)^\a u\right],
 \end{equation}
 where $s\in(1,2)$ is a penalization constant of weather at over $u$, $\lambda_1$ is the dew point measure defined by the vapor pressure $0.6108*\exp\left\{\frac{17.27T_d}{T_d+237.3}\right\}$, where dew point temperature $T_d$ in defined in Celsius and $\lambda_2$ is the speed of wind such that $(\lambda_1+\lambda_2)>1$. 
 
 \begin{as}\label{asf4}
 	$\bm{\sigma}^*(s,\mathbf{W},\mathbf{Z})$ is a positive, finite part of the diffusion component in Equation (\ref{f1}) which satisfies Assumptions \ref{asf0} and \ref{asf1} and is defined as 
 	\begin{multline}\label{f5}
 	\bm{\sigma}^*(s,\mathbf{W},\mathbf{Z})=\mathbf{p}(s,\mathbf{Z})+\mathbf{A}(s,\mathbf{W})+
 	\mathfrak{B}(\mathbf{Z})+\mathbf{Z}_e(s)\\
 	+\rho_1 \mathbf{p}^T(s,\mathbf{Z})\mathbf{A}(s,\mathbf{W})+
 	\rho_2\mathbf{A}^T(s,\mathbf{W})\mathfrak{B}(\mathbf{Z})+
 	\rho_3\mathfrak{B}^T(\mathbf{Z})\mathbf{p}(s,\mathbf{Z}),
 	\end{multline}
 	where $\rho_j\in(-1,1)$ is the $j^{th}$ correlation coefficient for $j=1,2,3$, and $\mathbf{A}^T,\mathfrak{B}^T$ and $\mathbf{p}^T$ are the transposition of $\mathbf{A},\mathfrak{B}$ and $\mathbf{p}$ which satisfy all conditions with Equations (\ref{f3}) and (\ref{f4}). As the ergodic function $\mathbf{Z}_e$ comes from nature, a team does not have any control on it and its correlation coefficient with other terms in Equation (\ref{f5}) are assumed to be zero.
 \end{as}

The randomness $\hat{\bm{\sigma}}$ of Equation (\ref{f2}) comes from type of skill of a soccer star of  the opposition team. There are mainly two main types of players: dribblers and tacklers, and free kickers. Free kickers  have two components, the speed of the ball after their kick $s\in\mathbb{R}_+^{(I\times I')\times t}$ in miles per hour and the curvature of the bowl path measure by the dispersion from the straight line connecting the goal keeper and the striker measured by $x\in\mathbb{R}_+^{(I\times I')\times t}$ inches. Define a payoff function $A_1(s,x,G):\mathbb{R}_+^{(I\times I')\times t}\times \mathbb{R}_+^{(I\times I')\times t}\times[0,1]\ra \mathbb{R}^{2(I\times I')\times t}$ such that, at time $s$ the expected payoff after guessing a ball right is $\E_s A_1(s,x,G)$, where $G$ is a guess function such that, if a player $i$ guesses the curvature and speed of the ball after an opposition player kicks then $G=1$ and if player $i$ does not then $G=0$, and if player $i$ partially guesses then, $G\in(0,1)$. 

On the other hand, there is a payoff function $A_2$ for a dribbler such that $A_2(s,x,\theta_1,G):\mathbb{R}_+^{(I\times I')\times t}\times \mathbb{R}_+^{(I\times I')\times t}\times[-k\pi,k\pi]\times[0,1]\ra\mathbb{R}^{2(I\times I')\times t} $, where $\theta_1$ is the angle between the beginning and end points when an opposition player start dribbling and end it after player $i$ gets the ball and $k\in\mathbb N$.  The expected payoff to score a goal at time $s$ when the opposition player is a dribbler is $\E_s A_2(s,x,\theta_1,G)$. Finally, an tackler's payoff function is $A_3(s,x,\theta_1,\theta_2,G):\mathbb{R}_+^{(I\times I')\times t}\times \mathbb{R}_+^{(I\times I')\times t}\times[-k\pi,k\pi]\times[-k,k\pi]\times[0,1]\ra\mathbb{R}^{2(I\times I')\times t}$, where $\theta_2$ is the allowable tackle movement in terms of angle when they do either of block, poke, slide tackles at time $s$ as $\E_s A_3(s,x,\theta_1,\theta_2,G)$. If $\theta_2$ is more than $k\pi$, the opposition player gets a foul or a yellow card. As player $i$ does not know who is what type of opposition they are going to face at a certain time during a game, their total expected payoff function at time $s$ is $\mathcal{A}(s,x,\theta_1,\theta_2,G)=\wp_1 \E_s A_1(s,x,G)+\wp_2 \E_s A_2(s,x,\theta_1,G)+\wp_3 \E_s A_3(s,x,\theta_1,\theta_2,G)$, where for $j\in\{1,2,3\}$, $\wp_j$ is the probability of each of a dribbler, a tackler and a free kicker with $\wp_1+\wp_2+\wp_3=1$. Therefore, $\mathcal{A}(s,x,\theta_1,\theta_2,G):\mathbb{R}_+^{(I\times I')\times t}\times \mathbb{R}_+^{(I\times I')\times t}\times[-k\pi,k\pi]\times[-k,k\pi]\times[0,1]\ra\mathbb{R}^{2(I\times I')\times t}$.

\section{Main results}
The components of stochastic differential games under $\sqrt(8/3)$-LQG with a continuum of states with dibbling and the passing function with finite actions are following:
\begin{itemize}
	\item $I$ be a non-empty finite set of players, $\mathcal F_s^{\mathbf Z}$ be the filtration of goal $\mathbf{Z}$ with the sample space $\Omega$.
	\item A finite set of actions at time $s$ for player $i$ such that $a_s^i\in A^i$ for all $i\in I$.
	\item A discount rate $\rho_s^i\in(0,1)$ for player $i\in I$ with the constant weight $\a^i\in\mathbb R$.
	\item The bounded objective function $\mathbf OB_\a^i$ expressed in Equation (\ref{f0}) is Borel measurable. Furthermore, the system has a goal dynamics expressed in the Equation (\ref{f1}).
	\item The game must be on a $\sqrt{8/3}$-LQG surface on $2$-sphere $\mathbb{S}^{(I\times I')\times t}$ with the Riemann metric tensor $e^{\sqrt{8/3}k(l)}d\mathbf {Z}\otimes d\widehat{\mathbf{Z}}$, where $k$ is some variant of GFF such that $k:\mathbb{S}^{(I\times I')\times t}\ra\mathbb R^{I\times I'}$.
	\item For $\varepsilon>0$ there exists a transition function from time $s$ to $s+\varepsilon$ expressed as $\Psi_{s,s+\varepsilon}^i(\mathbf Z):\Omega\times\mathbb{S}\times\mathbb R_+^{I\times I'}\times\prod_i A^i\ra\Delta(\Omega\times \mathbb S\times\mathbb R_+^{I\times I'})$ which is Borel-measurable.
	\item For a fixed quantum knowledge space $(\Omega,\mathcal F_s^{\bf Z},\mathbb S, I,\mathcal E)$, for player $i$ a dribbling and passing function $p^i$ is a mapping $p^i:\Omega\times\mathbb S\ra\Delta(\Omega\times \mathbb S)$ which is $\sigma$-measurable and the equivalence relationship has some measure in $2$-sphere.
\end{itemize}
The game is played in continuous time. If $\mathbf{Z}_{I\times I'}\subset \mathcal Z\subset(\Omega\times \mathbb S)$ be the goal condition after the start of $(M+1)^{th}$ game and player $i$ select an action profile at time $s$ such that $a_s^i\in\prod_i A^i$, then for $\varepsilon>0$, $\Psi_{s,s+\varepsilon}(\mathbf Z,a_s^i)$ is the conditional probability distribution of the next stage of the game. A stable strategy for a soccer player $i$ is a behavioral strategy that depends on the goal condition at time $s$. Therefore, we can say it is Borel measurable mapping associates with each goal $\mathbf Z\subset\Omega$ a probability distribution on the set $A^i$.

\begin{definition}
A stochastic differential game is purely quantum if it has countable orbits on $2$-sphere. In other words,
\begin{itemize}
	\item For each goal condition $\mathbf{Z}_{I\times I'}\subset \mathcal Z\subset(\Omega\times \mathbb S)$, every action profile of player $i$ starts at time $s$, $a_s^i\in A^i$ with the $\sqrt{8/3}$-LQG measure for a very small space on $\mathbb S$ defined as $e^{\sqrt{8/3}k(l)}$, $i^{th}$ player's transition function $\Psi_{s,s+\varepsilon}^i(\mathbf Z)$ is a purely quantum measure. Define
	\[
	\mathcal Q(\mathbf Z):=\left\{\mathbf Z'\in(\Omega\times\mathbb S)\big|\exists a_s^i\in A^i,e^{\sqrt{8/3}k(l)}\neq 0, \Psi_{s,s+\varepsilon}(\mathbf Z'|\mathbf Z,a_s^i)>0 \right\}.
	\]
	\item For each goal condition $\mathbf{Z}_{I\times I'}\subset \mathcal Z\subset(\Omega\times \mathbb S)$, the set
	\[
	\mathcal Q^{-1}(\mathbf Z):=\left\{\mathbf Z'\in(\Omega\times\mathbb S)\big|\exists a_s^i\in A^i,e^{\sqrt{8/3}k(l)}\neq 0, \Psi_{s,s+\varepsilon}(\mathbf Z'|\mathbf Z,a_s^i)>0 \right\}
	\]
	is countable.
\end{itemize}
\end{definition}
For purely atomic games see \cite{hellman2019}.
\begin{prop}\label{fp1}
	A purely quantum game on $\mathbb S$ with the objective function expressed in the Equation (\ref{f0}) subject to the goal dynamics expressed in the Equation (\ref{f1}) in which the orbit equivalence relation is smooth admits a measurable stable equilibrium.
\end{prop}

We know a soccer game stops if rainfall is so heavy that it restricts the vision of the player, making it dangerous or the pitch becomes waterlogged due to the heavy rain. After a period of stoppage, the officials will determine the conditions with tests using the ball while the players wait off the field. Suppose, a match stops after time $\tilde t$ because of the rain. After that there are two possibilities: first, if the rain is heavy, the game will not resume; secondly, if the rain is not so heavy and stops after certain point of time then, after  getting water out of the field, the match might be resumed. Based on the severity of the rain and the equipment used to get the water out from the field, the match resumes for $(\tilde t,t-\varepsilon]$ where $\varepsilon\geq 0$. The importance of $\varepsilon$ is that, if the rain is very heavy,  $\varepsilon=t-\tilde t$ and on the other hand, for the case of very moderate rain, $\varepsilon=0$. Therefore, $\varepsilon\in[0,t-\tilde t]$.

\begin{definition}\label{fde3}
	For a probability space $(\Omega, \mathcal{F}_s^{\mathbf{Z}},\mathcal{P})$ with sample space $\Omega$, filtration at time $s$ of goal condition ${\mathbf{Z}}$ as $\{\mathcal{F}_s^{\mathbf{Z}}\}\subset\mathcal{F}_s$, a probability measure $\mathcal{P}$ and a Brownian motion for rain $\mathbf{B}_s$ with the form $\mathbf{B}_s^{-1}(E)$ such that for $s\in[\tilde t,t-\varepsilon]$, $E\subseteq\mathbb{R}$ is a Borel set. If $\tilde{t}$ is the game stopping time because of rain and ${b}\in \mathbb{R}$ is a measure of rain in millimeters then $\tilde t:=\inf\{s\geq 0|\ \mathbf{B}_s>{b}\}$. 
\end{definition}

\begin{defn}\label{fde4}
	Let $\delta_{\mathfrak s}:[\tilde t,t-\varepsilon]	\ra(0,\infty)$ be a $C^2(s\in[\tilde t,t-\varepsilon])$ time-process of a match such that, it replaces stochastic process by It\^o's Lemma. Then $\delta_{\mathfrak s}$ is a stochastic gauge of that match if  $\tilde t:=\mathfrak s+\delta_{\mathfrak s}$ is a stopping time for each $\mathfrak s\in[\tilde t,t-\varepsilon]$ and $\mathbf B_{\mathfrak s}>b$, where $\mathfrak{s}$ is the new time after resampling the stochastic interval $[\tilde t,t-\varepsilon]$ on $\sqrt{8/3}$-LQG surface.
\end{defn}

\begin{defn}\label{fde5}
	Given a stochastic time interval of the soccer game $\hat I=[\tilde t,t-\varepsilon]\subset\mathbb{R}$, a stochastic tagged partition of that match is a finite set of ordered pairs $\mathcal{D}=\{(\mathfrak s_i,\hat I_i):\ i=1,2,...,p\}$ such that $\hat I_i=[x_{i-1},x_i]\subset[\tilde t,t-\varepsilon]$, $\mathfrak s_i\in \hat I_i$, $\cup_{i=1}^p \hat I_i=[\tilde t,t-\varepsilon]$ and for $i\neq j$ we have $\hat I_i\cap \hat I_j=\{\emptyset\}$. The point $\mathfrak s_i$ is the tag partition of the stochastic time-interval $\hat I_i$ of the game.
\end{defn}

\begin{defn}\label{fde6}
	If $\mathcal{D}=\{(\mathfrak s_i,\hat I_i): i=1,2,...,p\}$ is a tagged partition of stochastic time-interval of the match $\hat I$ and $\delta_{\mathfrak s}$ is a stochastic gauge on $\hat I$, then $\mathcal D$ is a stochastic $\delta$-fine if $\hat I_i\subset \delta_{\mathfrak s}(\mathfrak s_i)$ for all $i=1,2,...,p$, where $\delta(\mathfrak s)=(\mathfrak s-\delta_{\mathfrak s}(\mathfrak s),\mathfrak s+\delta_{\mathfrak s}(\mathfrak s))$.
\end{defn}

For a tagged partition $\mathcal D$ in a stochastic time-interval $\hat I$, as defined in Definitions \ref{fde5} and \ref{fde6}, and a function $\tilde f:[\tilde t,t-\varepsilon]\times \mathbb{R}^{2(I\times I')\times \hat t}\times\Omega\times\mathbb S\ra\mathbb{R}^{(I\times I')\times \hat t}$ the Riemann sum of $\mathcal D$ is defined as 
\[
S(\tilde f,\mathcal D)=(\mathcal D_\delta)\sum \tilde 
f(\mathfrak s,\hat I,\mathbf W,\mathbf Z)=
\sum_{i=1}^p \tilde f(\mathfrak s_i,\hat I_i,\mathbf W,\mathbf Z),
\]
where $\mathcal{D}_\delta$ is a $\delta$-fine division of 
$\mathbb{R}^{(I\times I')\times \widehat U}$ 
with point-cell function 
$\tilde f(\mathfrak s_i,\hat I_i,\mathbf W,\mathbf Z)=
\tilde f(\mathfrak s_i,\mathbf W,\mathbf Z)\ell(\hat I_i)$, 
where $\ell$ is the length of the over interval and 
$\hat{t}=(t-\varepsilon)-\tilde{t}$ 
\citep{kurtz2004,muldowney2012}.

\begin{defn}\label{fde7}
	An integrable function $\tilde f(\mathfrak s,\hat I,\mathbf W,\mathbf Z)$ on $\mathbb{R}^{(I\times I')\times\hat t}$, with integral 
	\[
	\mathbf a=\int_{\tilde t}^{t-\varepsilon}\tilde{f}(\mathfrak s, \hat I,\mathbf W,\mathbf Z)
	\]
	is stochastic Henstock-Kurzweil type integrable on $\hat I$ if, for a given vector $\hat{\bm\epsilon}>0$, there exists a stochastic $\delta$-gauge in $[\tilde t,t-\varepsilon]$ such that for each stochastic $\delta$-fine partition $\mathcal D_\delta$ in $\mathbb R^{(I\times I')\times \hat t}$ we have, 
	\[
	\E_{\mathfrak s}\left\{\left|\mathbf a-(\mathcal D_\delta)\sum \tilde f(\mathfrak s,\hat I,\mathbf W,\mathbf Z)\right|\right\}<\hat{\bm\epsilon},
	\]
	where $\E_{\mathfrak s}$ is the conditional expectation on goal $\mathbf{Z}$ at sample time $\mathfrak s\in[\tilde t,t-\varepsilon]$ of a non-negative function $\tilde f$ after the rain stops.
\end{defn}

\begin{prop}\label{fp2}
	Define 
	\[
	\mathfrak h=\exp\left\{-
	\tilde\epsilon\E_{\mathfrak s}\left[\int_{\mathfrak s}^{\mathfrak s+
		\tilde\epsilon}\tilde f(\mathfrak s,\hat I,\mathbf W,\mathbf Z)\right]\right\}
	\Psi_{\mathfrak s}(\mathbf{Z})d\mathbf Z.
	\]
	If for a small sample time interval $[\mathfrak s,\mathfrak s+\tilde\epsilon]$, 
	\[
	\frac{1}{N_{\mathfrak s}}\int_{\mathbb R^{2(I\times I')\times\widehat t\times I}}\mathfrak h
	\]
	exists for a conditional gauge $\gamma=[\delta,\omega(\delta)]$, then the indefinite integral of $\mathfrak h$, 
	\[
	\mathbf H(\mathbb R^{2(I\times I')\times\hat t\times I})=\frac{1}{N_{\mathfrak s}}\int_{\mathbb R^{2(I\times I')\times\hat t\times I}}\mathfrak h
	\]
	exists as Stieltjes function in $ \mathbf E([\mathfrak s,\mathfrak s+\tilde\epsilon]
	\times \mathbb R^{2(I\times I')\times\hat t}\times \Omega\times\mathbb S\times\mathbb R^I)$ for all $N_{\mathfrak s}>0$.
\end{prop}

\begin{coro}\label{fc0}
	If $\mathfrak h$ is integrable on $\mathbb R^{2(I\times I')\times\hat t\times I}$ as in Proposition \ref{fp2}, then for a given small continuous sample time the interval $[\mathfrak s,\mathfrak s']$ with 
	$\tilde{\epsilon}=\mathfrak s'-\mathfrak s>0$, 
	there exists a $\gamma$-fine division $\mathcal D_\gamma$ in $\mathbb R^{2(I\times I')\times\hat t\times I}$ such that,
	\[
	|(\mathcal D_\gamma)\mathfrak h[\mathfrak s,\hat I,\hat I(\mathbf Z),\mathbf W,\mathbf Z]-\mathbf H(\mathbb R^{2(I\times I')\times\hat t\times I})|\leq\mbox{$\frac{1}{2}$}|\mathfrak s-\mathfrak s'|<\tilde{\epsilon},
	\]
	where $\hat I(\mathbf Z)$ is the interval of goal $\mathbf Z$ in $\mathbb R^{2(I\times I')\times\hat t\times I}$. This integral is a stochastic It\^o-Henstock-Kurtzweil-McShane-Feynman-Liouville type path integral in goal dynamics of a sample time after the beginning of a match after rain interruption.
\end{coro}

As after the rain stops, the environment of the field changes, giving rise to a different probability space $(\Omega,\mathcal F_{\mathfrak{s}}^{\hat{ \mathbf Z}},\mathcal P)$, where $\mathcal F_{\mathfrak{s}}^{\hat{ \mathbf Z}}$ is the new filtration process at time $\mathfrak s$ after rain. The objective function after rain becomes,
\begin{multline}\label{f8}
\overline{\mathbf {OB}}_\a^i:\widehat{\mathbf Z}_\a^i(\mathbf W,\mathfrak s)\\=h_{\tilde t}^{i*}+\max_{W_i\in W} \E_{\tilde t}\left\{\int_{\tilde t}^{t-\varepsilon}\sum_{i=1}^I\sum_{m=1}^M\exp(-\rho_{\mathfrak s}^im)\a^iW_i(\mathfrak s)h_0^i[s,w(\mathfrak s),z(\mathfrak s)]\bigg|\mathcal F_{\mathfrak s}^{\hat{ \mathbf Z}}\right\}d\mathfrak s.
\end{multline}
Furthermore, if the match starts at time $\tilde t$ after the stoppage of the match due to rain then Liouville like action function on goal dynamics after the match starts after the rain is
\begin{multline}\label{f9}
\mathcal{L}_{\tilde{t},t-\varepsilon}(\mathbf{Z})=
\int_{\tilde{t}}^{t-\varepsilon}\E_{\mathfrak{s}} \left\{\sum_{i=1}^{I}\sum_{m=1}^M
\exp(-\rho_{\mathfrak s}^im)\a^iW_i(\mathfrak s)h_0^i[\mathfrak s,w(\mathfrak s),z(\mathfrak s)] \right.\\
+\lambda_1[\Delta\mathbf Z(\mathfrak s,\mathbf W)-\bm\mu[\mathfrak s,\mathbf W(\mathfrak s),\mathbf Z(\mathfrak s,\mathbf W)]d\mathfrak s-\bm\sigma[\mathfrak s,\hat{\bm\sigma},\mathbf W(\mathfrak s),\mathbf Z(\mathfrak s,\mathbf W)]d\mathbf B(\mathfrak s)]\\
\left.\phantom{\int}
+\lambda_2 e^{\sqrt{8/3}k(l(\mathfrak s))}d\mathfrak s
\right\}.
\end{multline}
The stochastic part of the Equation (\ref{f1}) becomes ${\bm{\sigma}}$ as $\hat{\lambda}_1>\lambda_1$. Equation (\ref{f9}) follows Definition \ref{fde7} such that $\mathbf a=\mathcal{L}_{\tilde t, t-\varepsilon}(\mathbf Z)$ and it is integrable according to Corollary \ref{fc0}.

\begin{prop}\label{fp3}
	If a team's objective is to maximize Equation (\ref{f8}) subject to the goal dynamics expressed in the Equation (\ref{f1}) on the $\sqrt{8/3}$-LQG surface, such that Assumptions \ref{asf0}- \ref{asf4} hold with  Propositions \ref{fp0}-\ref{fp2} and Corollary \ref{fc0}, then after a rain stoppage under a continuous sample time, the weight of player $i$ is found by solving 
	\begin{multline*}
	\sum_{i=1}^{I}\sum_{m=1}^M\exp(-\rho_{\mathfrak s}^im)\a^ih_0^i[\mathfrak s,w(\mathfrak s),z(\mathfrak s)]\\
	+g_{\mathbf{Z}}[\mathfrak s,\mathbf{Z}(\mathfrak s,\mathbf W)] 
	\frac{\partial
		\{\bm\mu[\mathfrak s,\mathbf{W}(\mathfrak s),\mathbf{Z}(\mathfrak s,\mathbf W)]\} 
	}{\partial \mathbf{W}}
	\frac{\partial \mathbf{W}}{\partial W_i }\\ 
	+\mbox{$\frac{1}{2}$} 
	\sum_{i=1}^I\sum_{j=1}^I 
	\frac{\partial
		{\bm\sigma}^{ij}[\mathfrak s,\hat{\bm{\sigma}},\mathbf{W}(\mathfrak s),\mathbf{Z}(\mathfrak s,\mathbf W)] }
	{\partial \mathbf{W}}
	\frac{\partial \mathbf{W}}{\partial W_i }
	g_{Z_iZ_j}[\mathfrak s,\mathbf{Z}(\mathfrak s,\mathbf W)]=0,
	\end{multline*}
	with respect to $\a_i$, where the initial condition before the first kick on the soccer ball after the rain stops  is $\mathbf{Z}_{\tilde s}$. Furthermore, when $\a_i=\a_j=\a^*$ for all $i\neq j$ we get a closed form solution of the player weight as 
	\begin{eqnarray*}
	\a^* & = & 
	-\left[\sum_{i=1}^{I}\sum_{m=1}^M\exp(-\rho_{\mathfrak s}^im)\a^ih_0^i[\mathfrak s,w(\mathfrak s),z(\mathfrak s)]\right]^{-1}\\
	& & \times\left[\frac{\partial g[\mathfrak s,\mathbf{Z}(\mathfrak s,\mathbf W)]}{\partial{\mathbf{Z}}}  
	\frac{\partial
		\{\bm{\mu}[\mathfrak s,\mathbf{W}(\mathfrak{s}),\mathbf{Z}(\mathfrak s,\mathbf W)]\} }{\partial \mathbf{W}}
	\frac{\partial \mathbf{W}}{\partial W_i } \right.\\ 
	& & \left.
	+\mbox{$\frac{1}{2}$}
	\sum_{i=1}^I\sum_{j=1}^I
	\frac{\partial {\bm\sigma}^{ij}[\mathfrak s,\hat{\bm{\sigma}},\mathbf{W}(\mathfrak s),
		\mathbf{Z}(\mathfrak s,\mathbf W)]}{\partial \mathbf{W}}
	\frac{\partial \mathbf{W}}{\partial W_i }
	\frac{\partial^2 g[\mathfrak s,\mathbf{Z}(\mathfrak s,\mathbf W)]}{\partial Z_i\partial Z_j} \right],
	\end{eqnarray*}
	where  function $g\left[\mathfrak s,\mathbf{Z}(\mathfrak s,\mathbf W)\right]\in C_0^2\left([\tilde t,t-\varepsilon]\times \mathbb R^{2(I\times I')\times \hat t}\times \mathbb{R}^I\right)$ with $\mathbf{Y}(\mathfrak s)=g\left[\mathfrak s,\mathbf{Z}(\mathfrak s,\mathbf W)\right]$ is a positive, non-decreasing penalization function vanishing at infinity which substitutes for the goal dynamics such that, $\mathbf{Y}(\mathfrak s)$ is an It\^o process.
\end{prop}

\section{Proofs}

\subsection{Proof of Proposition \ref{fp1}}
Consider a stochastic differential game $\left[\Omega\times\mathbb S,k,(\a^i)_{i\in I},(\rho_s^i)_{i\in I}, \left(\Psi_{s,s+\varepsilon}^i\right)_{i\in I},\right.\\\left.(p^i)_{i\in I}\right]$. Let us define the goal space $\mathcal Z$ is a countable set such that all common knowledge equivalence relations can be represented cardinally in $\mathcal Z$. First we define a quantifier free formula whose free variables are beliefs, payoffs of a player, actions based on the goal condition at time $s$ of stochastic differential game on $2$-sphere with goal space $\mathcal Z$ such that at time $s$ for a given goal condition $\mathbf Z_s$ and beliefs, the strategy polygon $\be_{k_i}$ has a Lefschetz-Hopf fixed point on $\sqrt{8/3}$-LQG action space. Assume $z_1,z_2,z_3,z_4$ are indices in $\mathcal Z$, $i,j$ are players in $I$, $a_s^i$ is player $i$'s action at time $s$ in $A^i$, and the action profile at time $s$ defined as $a_s:=\left(a_1^s(p^i), a_2^s(p^i),...,a_\kappa^s(p^i)\right)$ in $\prod_iA^i$, where $p^i$ is the dribbling and passing function.

For $j\in I$, $p^j:\Omega\times\mathbb S\ra\Delta(\Omega\times \mathbb S)$ and $z_3,z_4\in\mathcal Z$ the variable $\eta_{z_3,z_4,p^j}^j(s)$ is player $j$'s belief about the goal condition $z_3$ at time $s+\varepsilon$ while at the goal condition $z_4$ at time $s$. For $j\in I$, $z_2\in\mathcal Z$, $\mathfrak B(\mathbf Z)\in\mathbb R_+^{(I\times I')\times t}$ and $a_s\in\prod_j A^j$, the variable $\omega_{z_2,a_s,\mathfrak B}^j(s)$ is defined as player $j$'s payoff at time $s$ with goal condition $z_2$ given their action profile and the expectation of goals based on whether their team is playing a day match or a day-night match defined by the function $\mathfrak B$ in the Equation (\ref{f3}) in the previous section. Finally, For $j\in I$,  $z_2\in\mathcal Z$, $\mathbf{Z}_e:[0,t]\ra\mathbb{R}$ and $a_s^j\in A^j$, the variable $\a_{z_2,a_s^j,\mathbf{Z}_e}^j(s)$ is the weight of player $j$ puts at time $s$  on their action $a_s^j$ when the goal condition is $z_2$ and at dew condition on the field $\mathbf{Z}_e$ defined in the Equation (\ref{f4}).

For $i\in I$ and for a mixed action $\a_{a_s^i}\in\Delta A^i$ define a function 
\[
\tau^i(\a_{a_s^i})=\left(\sum_{a_s^i\in A^i}\a_{a_s^i}=1\right)\bigwedge_{a_s^i\in A^i}(\a_{a_s^i}\geq 0),
\]
and for $z_1,z_2\in\mathcal Z$ define
\[
\gamma_{z_1,z_2,p^j}^j\left[\left(\eta_{z_3,z_4,p^j}^j(s)\right)_{z_3,z_4,p^j}\right]=\bigwedge_{z_3,p^j}\left[\eta_{z_3,z_1,p^j}^j(s)=\eta_{z_3,z_2,p^j}^j(s)\right],
\]
where the above argument means for a given dribbling and passing function $p^j$ player $j$ has the same belief at both goal condition $z_1$ and $z_2$. Therefore, the function
\begin{align*}
&\tau\left\{\left[\a_{z_2,a_s^i,\mathbf{Z}_e}^i(s)\right],\left[\eta_{z_3,z_4,p^i}^i(s)\right]\right\}\\&=\left\{\bigwedge_i\bigwedge_{z_2}\tau^i\left[\a_{z_2,a_s^i,\mathbf{Z}_e}^i(s)\right]\right\}\bigwedge\left\{\bigwedge_i\bigwedge_{z_1,z_2}\gamma_{z_1,z_2,p^i}^i\left[\left(\eta_{z_3,z_4,p^j}^j(s)\right)_{z_3,z_4,p^j}\right]\right.\\&\hspace{1cm}\left.\ra\bigwedge_{a_s^i}\left[\a_{z_1,a_s^i,\mathbf{Z}_e}^i(s)=\a_{z_2,a_s^i,\mathbf{Z}_e}^i(s)\right]\right\},
\end{align*}
exists iff mixed actions are utilized at every goal condition (i.e. $\bigwedge_i\bigwedge_{z_2}\tau^i[\a_{z_2,a_s^i,\mathbf{Z}_e}^i(s)]$), and strategies are measurable with respect to player $i$'s knowledge of the game. Now for a transition function of player $i$ in the time interval $[s,s+\varepsilon]$ defined as $\Psi_{s,s+\varepsilon}^i(\mathbf Z)$ with the payoff of them at the beginning of time $s$ is $v_s^i\geq 0$. For $k$ on $\mathbb S$ define a state function
\begin{align}\label{f6}
&\widehat{\mathbf Z}\left\{v_s^i,\omega_{z_2,a_s,\mathfrak B}^i(s),\a_{z_2,a_s^i,\mathbf{Z}_e}^i(s),\Psi_{s,s+\varepsilon}^i(\mathbf Z)\right\}\notag\\&=\bigwedge_i\bigwedge_{z_1}\left[v_s^i=\sum_{a_s}\left(\prod_{i\in I}\a_{z_1,a_s^i,\mathbf{Z}_e}^i(s)\right)\left(\omega_{z_2,a_s,\mathfrak B}^i(s)+(1-\theta)\Psi_{s,s+\varepsilon}^i(\mathbf Z) v_{s+\varepsilon}^i\right)\right],
\end{align}
such that for the goal condition $z_1\in\mathcal Z$, the payoffs under mixed action profile of the game is $v_s^i$, where $\theta\in(0,1)$ is a discount factor of this game. Furthermore, if we consider the objective function expressed in the Equation (\ref{f0}) subject to the goal dynamics expressed in the Equation (\ref{f1}) on the $\sqrt{8/3}$-LQG action space, the function defined in Equation (\ref{f6}) in time interval $[s,s+\varepsilon]$ would be
\begin{align*}
&\widehat{\mathbf Z}_{s,s+\varepsilon}\left\{v_s^i,\a_{z_2,a_s^i,\mathbf{Z}_e}^i(s),\Psi_{s,s+\varepsilon}^i(\mathbf Z)\right\}\notag\\&=\bigwedge_i\bigwedge_{z_1}\biggr\{v_s^i=\sum_{a_s}\left(\prod_{i\in I}\a_{z_1,a_s^i,\mathbf{Z}_e}^i(s)\right)\\&\left[\E_s\int_s^{s+\varepsilon}\bigg(\sum_{i=1}^I\sum_{m=1}^M\exp(-\rho_s^im)\a^iW_i(s)h_0^i[s,w(s),z(s)]\right.\\&\left.+\lambda_1[\Delta\mathbf Z(\nu,\mathbf W)-\bm\mu[\nu,\mathbf W(\nu),\mathbf Z(\nu,\mathbf W)]d\nu-\bm\sigma[\nu,\hat{\bm\sigma},\mathbf W(\nu),\mathbf Z(\nu,\mathbf W)]d\mathbf B(\nu)]\right.\\&\hspace{4cm}\left.+\lambda_2e^{\sqrt{8/3}k(l)}d\nu\bigg)\right]\biggr\},
\end{align*}
where the expression inside the bracket $[.]$ is the quantum Lagrangian with two non-negative time independent Lagrangian multipliers $\lambda_1$ and $\lambda_2$ where the second stands for $\sqrt{8/3}$-LQG surface. Define
\begin{align}\label{f7}
&\widetilde{\mathbf Z}\left\{v_s^i,\omega_{z_2,a_s,\mathfrak B}^i(s),\a_{z_2,a_s^i,\mathbf{Z}_e}^i(s),\Psi_{s,s+\varepsilon}^i(\mathbf Z)\right\}\notag\\&=\bigwedge_i\bigwedge_{z_1}\bigwedge_{d\in A^i}\left[v_s^i\geq\sum_{z_2}\sum_{a_s}\left(\prod_{j\neq i}\a_{z_1,a_s^j,\mathbf{Z}_e}^j(s)\right)\right.\notag\\&\left.\times\left(\omega_{z_2,d,a_s,\mathfrak B}^j(s)+(1-\theta)\Psi_{s,s+\varepsilon}^j(\mathbf Z) v_{s+\varepsilon}^j\right)\right],
\end{align}
a goal condition $z_1\in\mathcal Z$ such that with no deviation in the game with continuation payoff $v_{s+\varepsilon}^j$ gets the payoff no more than the payoff at time $s$ or $v_s^i$. Therefore, for time interval $[s,s+\varepsilon]$ the goal condition should be,
\begin{align*}
&\widetilde{\mathbf Z}_{s,s+\varepsilon}\left\{v_s^i,\a_{z_2,a_s^i,\mathbf{Z}_e}^i(s),\Psi_{s,s+\varepsilon}^i(\mathbf Z)\right\}\notag\\&=\bigwedge_i\bigwedge_{z_1}\bigwedge_{d\in A^i}\biggr\{v_s^i\geq\sum_{z_2}\sum_{a_s}\left(\prod_{j\neq i}\a_{z_1,d,a_s^j,\mathbf{Z}_e}^j(s)\right)\\&\left[\E_s\int_s^{s+\varepsilon}\bigg(\sum_{i=1}^I\sum_{m=1}^M\exp(-\rho_s^jm)\a^jW_j(s)h_0^j[s,w(s),z(s)]\right.\\&\left.+\lambda_1[\Delta\mathbf Z(\nu,\mathbf W)-\bm\mu[\nu,\mathbf W(\nu),\mathbf Z(\nu,\mathbf W)]d\nu-\bm\sigma[\nu,\hat{\bm\sigma},\mathbf W(\nu),\mathbf Z(\nu,\mathbf W)]d\mathbf B(\nu)]\right.\\&\hspace{4cm}\left.+\lambda_2e^{\sqrt{8/3}k(l)}d\nu\bigg)\right]\biggr\}.
\end{align*}
Finally, define
\begin{align*}
&\widetilde{\mathbf Z}_{s,s+\varepsilon}^*\left\{v_s^i,\omega_{z_2,a_s,\mathfrak B}^i(s),\a_{z_2,a_s^i,\mathbf{Z}_e}^i(s),\Psi_{s,s+\varepsilon}^i(\mathbf Z)\right\}\\&=\bigwedge_{j\in I}\bigwedge_{z_1\in\mathcal Z}\tau^i\left[\a_{z_2,a_s^i,\mathbf{Z}_e}^i(s)\right]\wedge\widehat{\mathbf Z}_{s,s+\varepsilon}\left\{v_s^i,\a_{z_2,a_s^i,\mathbf{Z}_e}^i(s),\Psi_{s,s+\varepsilon}^i(\mathbf Z)\right\}\\&\hspace{1cm}\wedge \widetilde{\mathbf Z}_{s,s+\varepsilon}\left\{v_s^i,\a_{z_2,a_s^i,\mathbf{Z}_e}^i(s),\Psi_{s,s+\varepsilon}^i(\mathbf Z)\right\},
\end{align*}
such that each player is playing a probability distribution in each goal condition subject to a goal dynamics, dew condition of the field, whether the match is a day or day-night match, and the strategy profiles in goal condition $z_1\in\mathcal Z$ give the right payoff using equilibrium strategies on $\sqrt{8/3}$-LQG action space. By \cite{parthasarathy1972} we know that, a stochastic game with countable states has a stationary equilibrium and by Theorem $4.7$ of \cite{hellman2019} we conclude that, the system admits a stable equilibrium under $2$-sphere. $\square$

\subsection{Proof of Proposition \ref{fp2}}

Define a gauge $\gamma=[\delta,\omega(\delta)]$ for all possible combinations of a $\delta$ gauge in $[\tilde t,t-\varepsilon]\times \mathbb R^{2(I\times I')\times\hat t}\times \Omega$ and $\omega(\delta)$-gauge in $\mathbb R^{2(I\times I')\times\hat t\times I}$ such that it is a cell in $[\tilde t,t-\varepsilon]\times \mathbb R^{2(I\times I')\times\hat t}\times \Omega\times\mathbb S\times\mathbb R^{2(I\times I')\times\hat t\times I}$, where $\omega(\delta):\mathbb{R}^{2(I\times I')\times\hat t\times I}\ra(0,\infty)^{2(I\times I')\times\hat t\times I}$ is at least a $C^1$ function. The reason behind considering $\omega(\delta)$ as a function of $\delta$ is because, after rain stops, if the match proceeds on time $s$ then we can get a corresponding sample time $\mathfrak{s}$ and a player has the opportunity to score a goal. Let $\mathcal D_\gamma$ be a stochastic $\gamma$-fine in cell $\mathbf E$ in $[\tilde t,t-\varepsilon]\times \mathbb R^{2(I\times I')\times\hat t}\times \Omega\times\mathbb S\times\mathbb R^I$. For any $\bm{\varepsilon}>0$ and for a $\delta$-gauge in $[\tilde t,t-\varepsilon]\times \mathbb R^{2(I\times I')\times\hat t}\times \Omega\times\mathbb S$ and $\omega(\delta)$-gauge in $\mathbb R^{2(I\times I')\times\hat t\times I}$ choose a $\gamma$ so that 
\[
\left|\frac{1}{N_{\mathfrak s}}(\mathcal D_\gamma)\sum \mathfrak h-\mathbf H(\mathbb R^{2(I\times I')\times\hat t\times I})\right|<\mbox{$\frac{1}{2}$}|\mathfrak s-\mathfrak s'|,
\]
where $\mathfrak s'=\mathfrak s+\tilde{\epsilon}$. Assume two disjoint sets $E^a$ and $E^b=[\mathfrak s,\mathfrak s+\tilde{\epsilon}]\times \mathbb R^{2(I\times I')\times\hat t}\times \Omega\times\mathbb S\times\mathbb \{R^I\setminus E^a\}$ such that $E^a\cup E^b=E$ . As the domain of $\tilde f$ is a $2$-sphere, Theorem $3$ in \cite{muldowney2012} implies there is a gauge $\gamma_a$ for set $E^a$ and a gauge $\gamma_b$ for set $E^b$ with $\gamma_a\prec\gamma$ and $\gamma_b\prec\gamma$, so that both the gauges conform in their respective sets. For every $\delta$-fine in $[\mathfrak s,\mathfrak s']\times\mathbb{R}^{2(I\times I')\times\hat t}\times\Omega\times\mathbb S$ and a positive 
$\tilde{\epsilon}=|\mathfrak s-\mathfrak s'|$, 
if a $\gamma_a$-fine division $\mathcal D_{\gamma_a}$ is of the set $E^a$ and $\gamma_b$-fine division $\mathcal D_{\gamma_b}$ is of the set $E^b$, then by the restriction axiom we know that $\mathcal D_{\gamma_a}\cup\mathcal D_{\gamma_b}$ is a $\gamma$-fine division of $E$. Furthermore, as $E^a\cap E^b=\emptyset$
\begin{align}
\mbox{$\frac{1}{N_{\mathfrak s}}$}\left(\mathcal D_{\gamma_a}\cup\mathcal D_{\gamma_b}\right)\sum\mathfrak h=\mbox{$\frac{1}{N_{\mathfrak s}}$}\left[(\mathcal D_{\gamma_a})\sum \mathfrak h+(\mathcal D_{\gamma_b})\sum\mathfrak h\right]=\a+\be.\notag
\end{align}
Let us assume that for every $\delta$-fine we can subdivide the set $E^b$ into two disjoint subsets $E_1^b$ and $E_2^b$ with their $\gamma_b$-fine divisions given by $\mathcal D_{\gamma_b}^1$ and $\mathcal D_{\gamma_b}^2$, respectively. Therefore, their Riemann sum can be written as $\be_1=\frac{1}{N_{\mathfrak s}}(\mathcal D_{\gamma_b}^1)\sum\mathfrak h$ and $\be_2=\frac{1}{N_{\mathfrak s}}(\mathcal D_{\gamma_b}^2)\sum\mathfrak h$, respectively. Hence, for a small sample time interval $[\mathfrak s,\mathfrak s']$,
\[
\big|\a+\be_1-\mathbf H(\mathbb R^{2(I\times I')\times\hat t\times I})\big|\leq\mbox{$\frac{1}{2}$}|\mathfrak s-\mathfrak s'|
\]
and
\[
\big|\a+\be_2-\mathbf H(\mathbb R^{2(I\times I')\times\hat t\times I})\big|\leq\mbox{$\frac{1}{2}$}|\mathfrak s-\mathfrak s'|.
\]
Therefore,
\begin{eqnarray}\label{fCauchy}
|\be_1-\be_2| & = & 
\left|\left[\a+\be_1-\mathbf H(\mathbb R^{2(I\times I')\times\hat t\times I})\right]-
\left[\a+\be_2-\mathbf H(\mathbb R^{2(I\times I')\times\hat t\times I})\right]\right|\notag \\
& \leq & \left|\a+\be_1-\mathbf H(\mathbb R^{2(I\times I')\times\hat t\times I})\right|+
\left|\a+\be_2-\mathbf H(\mathbb R^{2(I\times I')\times\hat t\times I})\right| \notag \\
& \leq & |\mathfrak s-\mathfrak s'|.
\end{eqnarray}
Equation (\ref{fCauchy}) implies that the Cauchy integrability of $\mathfrak h$ is satisfied, and 
\[
\mathbf H(\mathbb R^{2(I\times I')\times\hat t\times I})=
\frac{1}{N_{\mathfrak s}}\int_{\mathbb R^{2(I\times I')\times\hat t\times I}}\mathfrak h.
\]
Now consider two disjoint set $M^1$ and $M^2$ in $\mathbb R^{2(I\times I')\times\hat t\times I}$ such that $M=M^1\cup M^2$ with their corresponding integrals $\mathbf H(M^1), \mathbf H(M^2)$, and $\mathbf H(M)$. Suppose $\gamma$-fine divisions of $M^1$ and $M^2$ are given by $\mathcal D_{\gamma_1}$ and $\mathcal D_{\gamma_2}$, respectively, with their Riemann sums for $\mathfrak h$ are $m_1$ and $m_2$. Equation (\ref{fCauchy}) implies, $\big|m_1-\mathbf H(M^1)\big|\leq\big|\mathfrak s-\mathfrak s'\big|$ and $\big|m_2-\mathbf H(M^2)\big|\leq\big|\mathfrak s-\mathfrak s'\big|$. Hence, $\mathcal D_{\gamma_1}\cup\mathcal D_{\gamma_2}$ is a $\gamma$-fine division of $M$. Let $m=m_1+m_2$ then Equation (\ref{fCauchy}) implies $\big|m-\mathbf H(M)\big|\leq |\mathfrak s-\mathfrak s'|$ and
\begin{eqnarray*}
	|[\mathbf H(M^1)+\mathbf H(M^2)]-\mathbf H(M)| & \leq &
	|m-\mathbf{H}(M)|+|m_1-\mathbf H(M^1)|+ \\
	& & |m_2-\mathbf{H}(M^2)| \\
	& \leq & 3|\mathfrak s-\mathfrak s'|.
\end{eqnarray*}
Therefore, $\mathbf H(M)=\mathbf H(M^1)+\mathbf H(M^2)$ and it is Stieljes. $\square$

\subsection{Proof of Proposition \ref{fp3}}

For a positive Lagrangian multipliers $\lambda_1$ and $\lambda_2$, with initial goal condition $\mathbf Z_{\tilde t}$ the goal dynamics are expressed in Equation (\ref{f1}) such that such that Definition \ref{fde7}, Propositions \ref{fp0}-\ref{fp2} and Corollary \ref{fc0} hold. Subdivide $[\tilde t,t-\varepsilon]$ into $n$ equally distanced small time-intervals $[\mathfrak s,\mathfrak s']$ such that $\tilde\epsilon\downarrow 0$, where $\mathfrak s'=\mathfrak s+\tilde\epsilon$.  For any positive $\tilde\epsilon$ and normalizing constant $N_{\mathfrak s}>0$, define a goal transition function as 
\[
\Psi_{\mathfrak s,\mathfrak s+\tilde\epsilon}(\mathbf{Z})=
\frac{1}{N_{\mathfrak s}}
\int_{\mathbb{R}^{2(I\times I')\times\hat t\times I}} \exp\bigg\{-\tilde\epsilon \mathcal{L}_{\mathfrak s,\mathfrak s+\tilde\epsilon}(\mathbf{Z}) \bigg\} \Psi_{\mathfrak s}(\mathbf{Z})d\mathbf{Z},
\]
where $\Psi_{\mathfrak t}(\mathbf{Z})$ is the goal transition function at the beginning of $\mathfrak t$, ${N_{\mathfrak s}}^{-1} d\mathbf{Z}$ is a finite Riemann measure which satisfies Proposition \ref{fp2}, and for $k^{th}$ sample time interval a goal transition function of $[\tilde t,t-\varepsilon]$ is,
\[
\Psi_{\tilde t,t-\tilde\varepsilon}(\mathbf{Z})=
\frac{1}{(N_{\mathfrak s})^n}
\int_{\mathbb{R}^{2(I\times I')\times\hat t\times I\times n}} \exp\bigg\{-\tilde\epsilon \sum_{k=1}^n\mathcal{L}_{\mathfrak s,\mathfrak s+\tilde\epsilon}^k(\mathbf{Z}) \bigg\}\Psi_0(\mathbf{Z}) \prod_{k=1}^n d\mathbf{Z}^k,
\]
with finite measure $\left(N_{\mathfrak s}\right)^{-n}\prod_{k=1}^{n}d\mathbf Z^k$ satisfying Corollary \ref{fc0} with its initial goal transition function after the rain stops as $\Psi_{\tilde t}(\mathbf Z)>0$ for all $n\in\mathbb N$.  Define $\Delta \mathbf{Z}(\nu)=\mathbf{Z}(\nu+d\nu)-\mathbf{Z}(\nu)$, then Fubuni's Theorem for the small interval of time $[\mathfrak s,\mathfrak s']$ with $\tilde\epsilon\downarrow 0$ yields,
\begin{multline*}
\mathcal{L}_{\mathfrak s,\mathfrak s'}(\mathbf{Z})
=\int_{\mathfrak s}^{\mathfrak s'}\E_{\nu} \left\{\sum_{i=1}^{I}\sum_{m=1}^M
\exp(-\rho_{\nu}^im)\a^iW_i(\nu)h_0^i[\nu,w(\nu),z(\nu)] \right.\\
+\lambda_1[\Delta\mathbf Z(\nu,\mathbf W)-\bm\mu[\nu,\mathbf W(\nu),\mathbf Z(\nu,\mathbf W)]d\nu-\bm\sigma[\nu,\hat{\bm\sigma},\mathbf W(\nu),\mathbf Z(\nu,\mathbf W)]d\mathbf B(\nu)]\\
\left.\phantom{\int}
+\lambda_2 e^{\sqrt{8/3}k(l(\nu))}d\nu
\right\}.
\end{multline*}
As we assume the goal dynamics have drift and diffusion parts, $\mathbf{Z} (\nu,\mathbf W)$ is an It\^o process and $\mathbf{W}$ is a Markov control measure of players. Therefore, there exists a smooth function $g[\nu,\mathbf{Z}(\nu,\mathbf W)]\in C^2\left([\tilde t,t-\varepsilon]\times\mathbb R^{2(I\times I')\times\hat t}\times \mathbb{R}^I\right)$ such that $\mathbf{Y}(\nu)=g[\nu,\mathbf{Z}(\nu,\mathbf W)]$ with $\mathbf{Y}(\nu)$ being an It\^o's process. Assume 
\begin{multline*}
g\left[\nu+\Delta\nu,\mathbf Z(\nu,\mathbf W)+\Delta\mathbf Z(\nu,\mathbf W)\right]=\\
\lambda_1[\Delta\mathbf Z(\nu,\mathbf W)-\bm\mu[\nu,\mathbf W(\nu),\mathbf Z(\nu,\mathbf W)]d\nu-\bm\sigma[\nu,\hat{\bm\sigma},\mathbf W(\nu),\mathbf Z(\nu,\mathbf W)]d\mathbf B(\nu)]\\
+\lambda_2 e^{\sqrt{8/3}k(l(\nu))}d\nu.
\end{multline*}
For a very small sample over-interval around $\mathfrak s$ with 
$\tilde\epsilon\downarrow 0$ generalized It\^o's Lemma gives,
\begin{multline*}
\tilde\epsilon\mathcal{L}_{\mathfrak s,\mathfrak s'}(\mathbf{Z})= 
\E_{\mathfrak s}\left\{\sum_{i=1}^{I}\sum_{m=1}^M
\tilde\epsilon\exp(-\rho_{\mathfrak s}^im)\a^iW_i(\mathfrak s)h_0^i[\mathfrak s,w(\mathfrak s),z(\mathfrak s)] \right.\\ +
\tilde\epsilon g[\mathfrak s,\mathbf{Z}(\mathfrak s,\mathbf W)]
+\tilde\epsilon g_{\mathfrak s}[\mathfrak s,\mathbf{Z}(\mathfrak s,\mathbf W)]+
\tilde\epsilon g_{\mathbf{Z}}[\mathfrak s,\mathbf{Z}(\mathfrak s,\mathbf W)] 
\{\bm{\mu}[\mathfrak s,\mathbf{W}(\mathfrak s),\mathbf{Z}(\mathfrak s,\mathbf W)]\} \\ 
+ \tilde\epsilon g_{\mathbf{Z}}[\mathfrak s,\mathbf{Z}(\mathfrak s,\mathbf W)]
{\bm{\sigma}}\left[\mathfrak s,\hat{\bm{\sigma}},
\mathbf{W}(\mathfrak s),\mathbf{Z}(\mathfrak s,\mathbf W)\right]
\Delta\mathbf{B}(\mathfrak s) \\ 
\left.\phantom{\sum}
+\mbox{$\frac{1}{2}$}\tilde\epsilon\sum_{i=1}^I\sum_{j=1}^I
{\bm\sigma}^{ij}
[\mathfrak s,\hat{\bm{\sigma}},\mathbf{W}(\mathfrak s),\mathbf{Z}(\mathfrak s,\mathbf W)]
g_{Z_iZ_j}[\mathfrak s,\mathbf{Z}(\mathfrak s,\mathbf W)]+o(\tilde\epsilon)\right\},
\end{multline*}
where ${\bm\sigma}^{ij}\left[\mathfrak s,\hat{\bm{\sigma}},\mathbf{W}(\mathfrak s),\mathbf{Z}(\mathfrak s,\mathbf W)\right]$ represents $\{i,j\}^{th}$ component of the variance-covariance matrix, $g_{\mathfrak s}=\partial g/\partial\mathfrak s$, 
$g_{\mathbf{Z}}=\partial g/\partial \mathbf{Z}$, 
$g_{Z_iZ_j}=\partial^2 g/(\partial Z_i\partial Z_j)$, 
$\Delta B_i\Delta B_j=\delta^{ij}\tilde\epsilon$, 
$\Delta B_i\tilde\epsilon=\tilde\epsilon\Delta B_i=0$, and 
$\Delta Z_i(\mathfrak s)\Delta Z_j(\mathfrak s)=\tilde\epsilon$, 
where $\delta^{ij}$ is the Kronecker delta function. 
As $\E_{\mathfrak s}[\Delta \mathbf{B}(\mathfrak s)]=0$, $\E_{\mathfrak s}[o(\tilde\epsilon)]/\tilde\epsilon\ra 0$ and for $\tilde\epsilon\downarrow 0$,
\begin{multline*}
\mathcal{L}_{\mathfrak s,\mathfrak s'}(\mathbf{Z})=
\sum_{i=1}^{I}\sum_{m=1}^M
\exp(-\rho_{\mathfrak s}^im)\a^iW_i(\mathfrak s)h_0^i[\mathfrak s,w(\mathfrak s),z(\mathfrak s)] \\ +
 g[\mathfrak s,\mathbf{Z}(\mathfrak s,\mathbf W)]
+ g_{\mathfrak s}[\mathfrak s,\mathbf{Z}(\mathfrak s,\mathbf W)]+
 g_{\mathbf{Z}}[\mathfrak s,\mathbf{Z}(\mathfrak s,\mathbf W)] 
\{\bm{\mu}[\mathfrak s,\mathbf{W}(\mathfrak s),\mathbf{Z}(\mathfrak s,\mathbf W)]\} \\ 
+  g_{\mathbf{Z}}[\mathfrak s,\mathbf{Z}(\mathfrak s,\mathbf W)]
{\bm{\sigma}}\left[\mathfrak s,\hat{\bm{\sigma}},
\mathbf{W}(\mathfrak s),\mathbf{Z}(\mathfrak s,\mathbf W)\right]
\Delta\mathbf{B}(\mathfrak s) \\ 
\phantom{\sum}
+\mbox{$\frac{1}{2}$}\sum_{i=1}^I\sum_{j=1}^I
{\bm\sigma}^{ij}
[\mathfrak s,\hat{\bm{\sigma}},\mathbf{W}(\mathfrak s),\mathbf{Z}(\mathfrak s,\mathbf W)]
g_{Z_iZ_j}[\mathfrak s,\mathbf{Z}(\mathfrak s,\mathbf W)]+o(1).
\end{multline*}
There exists a vector  $\mathbf{\xi}_{(2(I\times I')\times\hat t\times I)\times 1}$  so that 
\[
\mathbf{Z}(\mathfrak s,\mathbf W)_{(2(I\times I')\times\hat t\times I)\times 1}=\mathbf{Z}(\mathfrak s',\mathbf W)_{(2(I\times I')\times\hat t\times I)\times 1}+\xi_{(2(I\times I')\times\hat t\times I)\times 1}. 
\]
Assume $|\xi|\leq\eta\tilde\epsilon [\mathbf{Z}^T(\mathfrak s,\mathbf W)]^{-1}$, then
\begin{multline*}
\Psi_{\mathfrak s}(\mathbf{Z})+
\tilde\epsilon\frac{\partial \Psi_{\mathfrak s}(\mathbf{Z})}{\partial \mathfrak s}+o(\tilde\epsilon)=
\frac{1}{N_{\mathfrak{s}}} 
\int_{\mathbb{R}^{2(I\times I')\times \hat t\times I}}
\left[\Psi_{\mathfrak s}(\mathbf{Z})+
\xi\frac{\partial \Psi_{\mathfrak u}(\mathbf{Z})}{\partial \mathbf{Z}}+
o(\tilde\epsilon)\right] \\\times
\exp\left\{-\tilde\epsilon\left[\sum_{i=1}^{I}\sum_{m=1}^M\exp(-\rho_{\mathfrak s}^im)\a^iW_i(\mathfrak s)h_0^i[\mathfrak s,w(\mathfrak s),z(\mathfrak s)+\xi]\right.\right. \\ 
+g[\mathfrak s,\mathbf{Z}(\mathfrak s',\mathbf W)+\xi]+ 
g_{\mathfrak s}[\mathfrak s,\mathbf{Z}(\mathfrak s',\mathbf W)+\xi] \\ 
+g_{\mathbf{Z}}[\mathfrak s,\mathbf{Z}(\mathfrak s',\mathbf W)+\xi]
\{\bm{\mu}[\mathfrak s,\mathbf{W}(\mathfrak s),\mathbf{Z}(\mathfrak s',\mathbf W)+\xi] \}\\ 
\left.\left.
+\mbox{$\frac{1}{2}$}\sum_{i=1}^I\sum_{j=1}^I
{\bm\sigma}^{ij}[\mathfrak s,\hat{\bm{\sigma}},\mathbf{W}(\mathfrak s),
\mathbf{Z}(\mathfrak s',\mathbf W)+\xi]
g_{Z_iZ_j}[\mathfrak s,\mathbf{Z}(\mathfrak s',\mathbf W)+\xi]\right]\right\}d\xi+
o(\tilde\epsilon^{1/2}).
\end{multline*}
Define a $C^2$ function 
\begin{eqnarray*}
	f[\mathfrak s,\mathbf W(\mathfrak s),\xi]
	& = & \sum_{i=1}^{I}\sum_{m=1}^M\exp(-\rho_{\mathfrak s}^im)\a^iW_i(\mathfrak s)h_0^i[\mathfrak s,w(\mathfrak s),z(\mathfrak s)+\xi] \\
	& & +g[\mathfrak s,\mathbf{Z}(\mathfrak s',\mathbf W)+\xi]+ 
	g_{\mathfrak s}[\mathfrak s,\mathbf{Z}(\mathfrak s',\mathbf W)+\xi] \\
	& & +g_{\mathbf{Z}}[\mathfrak s,\mathbf{Z}(\mathfrak s',\mathbf W)+\xi]
	\{\bm{\mu}[\mathfrak s,\mathbf{W}(\mathfrak s),\mathbf{Z}(\mathfrak s',\mathbf W)+\xi] \\ 
	& & +\mbox{$\frac{1}{2}$}\sum_{i=1}^I\sum_{j=1}^I
	{\bm\sigma}^{ij}[\mathfrak s,\hat{\bm{\sigma}},\mathbf{W}(\mathfrak s),\mathbf{Z}
	(\mathfrak s',\mathbf W)+\xi] \\ 
	& & \times g_{Z_iZ_j}[\mathfrak s,\mathbf{Z}(\mathfrak s',\mathbf W)+\xi].
\end{eqnarray*} 
Hence,
\begin{eqnarray*}
	\Psi_{\mathbf{s}}(\mathbf{Z})+
	\tilde\epsilon\frac{\partial \Psi_{\mathfrak s}(\mathbf{Z})}{\partial \mathfrak s} & = & 
	\frac{\Psi_{\mathfrak s}(\mathbf{Z})}{N_{\mathfrak s}} 
	\int_{\mathbb{R}^{2(I\times I')\times \hat t\times I}} 
	\exp\{-\tilde\epsilon f[\mathfrak s,\mathbf{W}(\mathfrak s),\xi]\}d\xi+ \\
	& &
	\hspace{-1cm}
	\frac{1}{N_{\mathfrak s}}\frac{\partial \Psi_{\mathfrak s}(\mathbf{Z})}{\partial \mathbf{Z}}
	\int_{\mathbb{R}^{2(I\times I')\times\hat{t}\times I}} 
	\xi\exp\{-\tilde\epsilon f[\mathfrak{s},\mathbf{W}(\mathfrak{s}),\xi]\}d\xi+ 
	o(\tilde\epsilon^{1/2}).
\end{eqnarray*}
For $\tilde\epsilon\downarrow0$, $\Delta \mathbf{Z}\downarrow0$
\begin{eqnarray*}
	f[\mathfrak s,\mathbf{W}(\mathfrak s),\xi] & = & 
	f[\mathfrak s,\mathbf{W}(\mathfrak s),\mathbf{Z}(\mathfrak s',\mathbf W)]\\
	& &+ \sum_{i=1}^{I}f_{Z_i}[\mathfrak s,\mathbf{W}(\mathfrak s),\mathbf{Z}(\mathfrak s,\mathbf W)][\xi_i-Z_i(\mathfrak s,\mathbf W)]\\
	& & 
	\hspace{-2.5cm}
	+\mbox{$\frac{1}{2}$}\sum_{i=1}^{I}\sum_{j=1}^{I}
	f_{Z_iZ_j}[\mathfrak s,\mathbf{W}(\mathfrak s),\mathbf{Z}(\mathfrak s',\mathbf W)]
	[\xi_i-Z_i(\mathfrak s',\mathbf W)][\xi_j-Z_j(\mathfrak s',\mathbf W)]+o(\tilde\epsilon).
\end{eqnarray*}
There exists a symmetric, positive definite and non-singular Hessian matrix $\mathbf{\Theta}_{[2(I\times I')\times\hat t\times I]\times [2(I\times I')\times\hat t\times I]}$ and a vector $\mathbf{R}_{(2(I\times I')\times\hat t\times I)\times 1}$ such that,
\begin{multline*}
\int_{\mathbb{R}^{2(I\times I')\times\hat t\times I}}
\exp\{-\tilde\epsilon f[\mathfrak s,\mathbf{W}(\mathfrak s),\xi]\}d\xi=\\
\sqrt{\frac{(2\pi)^{2(I\times I')\times\hat t\times I}}{\tilde\epsilon |\mathbf{\Theta}|}}
\exp\left\{-\tilde\epsilon f[\mathfrak s,\mathbf{W}(\mathfrak s),\mathbf{Z}(\mathfrak s',\mathbf W)]
+\mbox{$\frac{1}{2}$}\tilde\epsilon\mathbf{R}^T\mathbf{\Theta}^{-1}\mathbf{R}\right\},
\end{multline*}
where
\begin{multline*}
\int_{\mathbb{R}^{2(I\times I')\times\hat t\times I}} \xi 
\exp\{-\tilde\epsilon f[\mathfrak s,\mathbf{W}(\mathfrak s),\xi]\} d\xi \\
=\sqrt{\frac{(2\pi)^{2(I\times I')\times\hat t\times I}}{\tilde\epsilon |\mathbf{\Theta}|}} 
\exp\{-\tilde\epsilon f[\mathfrak s,\mathbf{W}(\mathfrak s),\mathbf{Z}(\mathfrak s',\mathbf W)]+
\mbox{$\frac{1}{2}$}\tilde\epsilon \mathbf{R}^T
\mathbf{\Theta}^{-1}\mathbf{R}\}\\\times
[\mathbf{Z}(\mathfrak s',\mathbf W)+\mbox{$\frac{1}{2}$} (\mathbf{\Theta}^{-1}\mathbf{R})].
\end{multline*}
Therefore
\begin{multline*}
\Psi_{\mathfrak s}(\mathbf{Z})+\tilde\epsilon 
\frac{\partial \Psi_{\mathfrak s}(\mathbf{Z})}{\partial \mathfrak{s}}=
\frac{1}{N_{\mathfrak s}} 
\sqrt{\frac{(2\pi)^{2(I\times I')\times\hat t\times I}}{\tilde\epsilon |\mathbf{\Theta}|}} \\\times
\exp\{-\tilde\epsilon f[\mathfrak s,\mathbf{W}(\mathfrak s),\mathbf{Z}(\mathfrak s',\mathbf W)]+
\mbox{$\frac{1}{2}$}\tilde\epsilon\mathbf{R}^T\mathbf{\Theta}^{-1}\mathbf{R}\}\\
\times\left\{\Psi_{\mathfrak s}(\mathbf{Z})+
[\mathbf{Z}(\mathfrak s',\mathbf W)+\mbox{$\frac{1}{2}$}(\mathbf{\Theta}^{-1} \mathbf{R})]
\frac{\partial \Psi_{\mathfrak{s}}(\mathbf{Z})}{\partial \mathbf{Z}}
\right\}+o(\tilde\epsilon^{1/2}).
\end{multline*}
Assuming $N_s=\sqrt{(2\pi)^{2(I\times I')\times\hat t\times I}/\left(\epsilon |\mathbf{\Theta}|\right)}>0$,
we get Wick rotated Schr\"odinger type equation as,
\begin{multline*}
\Psi_{\mathfrak s}(\mathbf{Z})+
\tilde\epsilon\frac{\partial \Psi_{\mathfrak s}(\mathbf{Z})}{\partial\mathfrak s} 
=\left\{1-\tilde\epsilon f[\mathfrak s,\mathbf{W}(s),\mathbf{Z}(\mathfrak s',\mathbf W)]+
\mbox{$\frac{1}{2}$}\tilde\epsilon\mathbf{R}^T\mathbf{\Theta}^{-1}\mathbf{R}\right\} \\\times
\left[\Psi_{\mathfrak s}(\mathbf{Z})+
[\mathbf{Z}(\mathfrak s',\mathbf W)+\mbox{$\frac{1}{2}$}(\mathbf{\Theta}^{-1} \mathbf{R})]
\frac{\partial \Psi_{\mathfrak{s}}(\mathbf{Z})}{\partial \mathbf{Z}}
\right]+o(\tilde\epsilon^{1/2}).
\end{multline*}
As $\mathbf{Z}(\mathfrak s,\mathbf W)\leq\eta\tilde\epsilon|\xi^T|^{-1}$, there exists $|\mathbf{\Theta}^{-1}\mathbf{R}|\leq 2 \eta\tilde\epsilon|1-\xi^T|^{-1}$ such that for 
$\tilde\epsilon\downarrow 0$ we have $\big|\mathbf{Z}(\mathfrak s',\mathbf W)+\mbox{$\frac{1}{2}$}\left(\mathbf{\Theta}^{-1}\ \mathbf{R}\right)\big|\leq\eta\tilde\epsilon$ and hence,
\[
\frac{\partial \Psi_{\mathfrak s}(\mathbf{Z})}{\partial \mathfrak s}=
\big[- f[\mathfrak s,\mathbf{W}(\mathfrak s),\mathbf{Z}(\mathfrak s',\mathbf W)]+
\mbox{$\frac{1}{2}$}\mathbf{R}^T\mathbf{\Theta}^{-1}\mathbf{R}\big]\Psi_{\mathfrak s}(\mathbf{Z}). 
\]
For $|\mathbf{\Theta}^{-1}\mathbf{R}|\leq 2 \eta\tilde\epsilon|1-\xi^T|^{-1}$ 
and at $\tilde\epsilon\downarrow 0$,
\[
\frac{\partial \Psi_{\mathfrak s}(\mathbf{Z})}{\partial \mathfrak s}=
-f[\mathfrak s,\mathbf{W}(\mathfrak s),\mathbf{Z}(\mathfrak s',\mathbf W)]
\Psi_{\mathfrak s}(\mathbf{Z}),
\]
and
\begin{equation}\label{f10}
-\frac{\partial }{\partial W_i}
f[\mathfrak s,\mathbf{W}(\mathfrak s),\mathbf{Z}(\mathfrak s',\mathbf W)] 
\Psi_{\mathfrak s}(\mathbf{Z})=0.
\end{equation}
In Equation (\ref{f10}),
$\Psi_{\mathfrak s}(\mathbf{Z})$ is the transition wave function and cannot be zero therefore,
\[
\frac{\partial }{\partial W_i}
f[\mathfrak s,\mathbf{W}(\mathfrak s),\mathbf{Z}(\mathfrak s',\mathbf W)]=0.
\]
We know, $\mathbf{Z}(\mathfrak s',\mathbf W)=\mathbf{Z}(\mathfrak s,\mathbf W)-\xi$ and for $\xi\downarrow 0$ as we are looking for some stable solution. Hence, $\mathbf{Z}(\mathfrak s',\mathbf W)$ can be replaced by $\mathbf{Z}(\mathfrak s,\mathbf W)$ and,
\begin{multline*}
f[\mathfrak s,\mathbf W(\mathfrak s),\mathbf Z(\mathfrak s,\mathbf W)]=
\sum_{i=1}^{I}\sum_{m=1}^M\exp(-\rho_{\mathfrak s}^im)\a^iW_i(\mathfrak s)h_0^i[\mathfrak s,w(\mathfrak s),z(\mathfrak s)] \\
+g[\mathfrak s,\mathbf{Z}(\mathfrak s,\mathbf W)]+ 
g_{\mathfrak s}[\mathfrak s,\mathbf{Z}(\mathfrak s,\mathbf W)] \\
+g_{\mathbf{Z}}[\mathfrak s,\mathbf{Z}(\mathfrak s,\mathbf W)]
\{\bm{\mu}[\mathfrak s,\mathbf{W}(\mathfrak s),\mathbf{Z}(\mathfrak s,\mathbf W)]\} \\
+\mbox{$\frac{1}{2}$}
\sum_{i=1}^I\sum_{j=1}^I{\bm\sigma}^{ij}
[\mathfrak s,\hat{\bm{\sigma}},\mathbf{W}(\mathfrak s),\mathbf{Z}(\mathfrak s,\mathbf W)]\ g_{Z_iZ_j}[\mathfrak s,\mathbf{Z}(\mathfrak s,\mathbf W)],
\end{multline*}
so that
\begin{multline}\label{f11}
\sum_{i=1}^{I}\sum_{m=1}^M\exp(-\rho_{\mathfrak s}^im)\a^ih_0^i[\mathfrak s,w(\mathfrak s),z(\mathfrak s)]\\
+g_{\mathbf{Z}}[\mathfrak s,\mathbf{Z}(\mathfrak s,\mathbf W)] 
\frac{\partial
	\{\bm\mu[\mathfrak s,\mathbf{W}(\mathfrak s),\mathbf{Z}(\mathfrak s,\mathbf W)]\} 
}{\partial \mathbf{W}}
\frac{\partial \mathbf{W}}{\partial W_i }\\ 
+\mbox{$\frac{1}{2}$} 
\sum_{i=1}^I\sum_{j=1}^I 
\frac{\partial
	{\bm\sigma}^{ij}[\mathfrak s,\hat{\bm{\sigma}},\mathbf{W}(\mathfrak s),\mathbf{Z}(\mathfrak s,\mathbf W)] }
{\partial \mathbf{W}}
\frac{\partial \mathbf{W}}{\partial W_i }
g_{Z_iZ_j}[\mathfrak s,\mathbf{Z}(\mathfrak s,\mathbf W)]=0.
\end{multline}
If in Equation (\ref{f11}) we solve for $\a_i$, we can get a solution of the weight attached to $W_i(\mathfrak s)$. In order to get a closed form solution we have to assume $\a_i=\a_j=\a^*$ for all $i\neq j$ which yields,
\begin{eqnarray}\label{f12}
\a^* & = & 
-\left[\sum_{i=1}^{I}\sum_{m=1}^M\exp(-\rho_{\mathfrak s}^im)\a^ih_0^i[\mathfrak s,w(\mathfrak s),z(\mathfrak s)]\right]^{-1}
\notag \\
& & \times\left[\frac{\partial g[\mathfrak s,\mathbf{Z}(\mathfrak s,\mathbf W)]}{\partial{\mathbf{Z}}}  
\frac{\partial
	\{\bm{\mu}[\mathfrak s,\mathbf{W}(\mathfrak{s}),\mathbf{Z}(\mathfrak s,\mathbf W)]\} }{\partial \mathbf{W}}
\frac{\partial \mathbf{W}}{\partial W_i } \right. \notag \\ 
& & \left.
+\mbox{$\frac{1}{2}$}
\sum_{i=1}^I\sum_{j=1}^I
\frac{\partial {\bm\sigma}^{ij}[\mathfrak s,\hat{\bm{\sigma}},\mathbf{W}(\mathfrak s),
	\mathbf{Z}(\mathfrak s,\mathbf W)]}{\partial \mathbf{W}}
\frac{\partial \mathbf{W}}{\partial W_i }
\frac{\partial^2 g[\mathfrak s,\mathbf{Z}(\mathfrak s,\mathbf W)]}{\partial Z_i\partial Z_j} \right].
\end{eqnarray}
Expression in Equation (\ref{f12}) is a unique closed form solution. 
The sign of $\a^*$ varies along with the signs of all the partial derivatives. $\square$

\section{Discussion}
In this paper we obtain a weight $\a^*$ for a soccer match with rain interruption. This coefficient tells us how to select a player to score goals at a certain position based on the condition of that match. It is a common practice to hide a new talented player until the $15$-minutes of the game. As the player is new opposition team has less information about them and as they are playing at the first time of the game, they have more energy than a player who is playing for last $75$ minutes. Our model will determine the weight associated with these players under more generalized and realistic conditions of the game.

We use a Feynman path integral technique to calculate $\a^*$. In the later part of the paper we focus on the more volatile environment after the rain stops. We assume that, after a rain stoppage the occurrence of each kick towards the goal strictly depends on the amount of rain at that sample time. If it is more than $b\in\mathbb R$ millimeters, then it is very hard to move with the ball on the field, and it is extremely difficult for a goal keeper to grip the ball which results not to resume the match again. Using It\^ o's lemma we define a $\delta_{\mathfrak s}$-gauge which generates a sample time $\mathfrak s$ instead of an actual time $s$ is assumed to follow a Wiener process. Furthermore, we assume the action space of a soccer player has a $\sqrt{8/3}$-LQG surface and, we construct a stochastic It\^o-Henstock-Kurzweil-McShane-Feynman-Liouville type path integral to solve for the optimal weight associated with them. As before rain, environment does not offer an extra moisture, a technically sound player does not need to predict the behavior of an opponent which is not true for the case of a match after a rain stoppage.

\bibliography{bib}

\begin{thebibliography}{49}

\bibitem[\protect\citeauthoryear{Ahamed}{2021a}]{ahamed2021}
\begin{barticle}[author]
\bauthor{\bsnm{Ahamed},~\bfnm{Faruque}\binits{F.}}
(\byear{2021}a).
\btitle{Macroeconomic Impact of Covid-19: A case study on Bangladesh}.
\bjournal{IOSR Journal of Economics and Finance (IOSR-JEF)}
\bvolume{12}
\bpages{2021}.
\end{barticle}
\endbibitem

\bibitem[\protect\citeauthoryear{Ahamed}{2021b}]{ahamed2021d}
\begin{barticle}[author]
\bauthor{\bsnm{Ahamed},~\bfnm{Faruque}\binits{F.}}
(\byear{2021}b).
\btitle{Determinants of Liquidity Risk in the Commercial Banks in Bangladesh}.
\bjournal{European Journal of Business and Management Research}
\bvolume{6}
\bpages{164--169}.
\end{barticle}
\endbibitem

\bibitem[\protect\citeauthoryear{Alam}{2021a}]{alam2021}
\begin{barticle}[author]
\bauthor{\bsnm{Alam},~\bfnm{Mohammad~Masud}\binits{M.~M.}}
(\byear{2021}a).
\btitle{A Probit Estimation of Urban Bases of Environmental Awareness: Evidence
  from Sylhet City, Bangladesh}.
\bjournal{arXiv preprint arXiv:2107.08342}.
\end{barticle}
\endbibitem

\bibitem[\protect\citeauthoryear{Alam}{2021b}]{alam2021output}
\begin{barticle}[author]
\bauthor{\bsnm{Alam},~\bfnm{Masud}\binits{M.}}
(\byear{2021}b).
\btitle{Output, Employment, and Price Effects of US Narrative Tax Changes: A
  Factor-Augmented Vector Autoregression Approach}.
\bjournal{arXiv preprint arXiv:2106.10844}.
\end{barticle}
\endbibitem

\bibitem[\protect\citeauthoryear{ALAM and HOSSAIN}{2018}]{alam2018}
\begin{barticle}[author]
\bauthor{\bsnm{ALAM},~\bfnm{MOHAMMAD~MASUD}\binits{M.~M.}} \AND
  \bauthor{\bsnm{HOSSAIN},~\bfnm{MD~KABIR}\binits{M.~K.}}
(\byear{2018}).
\btitle{POLICY OPTIONS ON SUSTAINABLE RESOURCE UTILIZATION AND FOOD SECURITY IN
  HAOR AREAS OF BANGLADESH: A THEORETICAL APPROACH}.
\bjournal{International Journal of Social, Political and Economic Research}
\bvolume{5}
\bpages{11--28}.
\end{barticle}
\endbibitem

\bibitem[\protect\citeauthoryear{Alam, Khondker and Molla}{2013}]{alam2013}
\begin{barticle}[author]
\bauthor{\bsnm{Alam},~\bfnm{Mohammad~Masud}\binits{M.~M.}},
  \bauthor{\bsnm{Khondker},~\bfnm{Rezai~Karim}\binits{R.~K.}} \AND
  \bauthor{\bsnm{Molla},~\bfnm{Mohammad~Shahansha}\binits{M.~S.}}
(\byear{2013}).
\btitle{Current Account Dynamics, Adjustment and Capital Mobility in
  Bangladesh}.
\bjournal{Global Disclosure of Economics and Business}
\bvolume{2}
\bpages{117--126}.
\end{barticle}
\endbibitem

\bibitem[\protect\citeauthoryear{Alam, Sultan and Afrin}{}]{alamanalyzing}
\begin{barticle}[author]
\bauthor{\bsnm{Alam},~\bfnm{Mohammad~Masud}\binits{M.~M.}},
  \bauthor{\bsnm{Sultan},~\bfnm{Muyed}\binits{M.}} \AND
  \bauthor{\bsnm{Afrin},~\bfnm{Sabiha}\binits{S.}}
\btitle{Analyzing Growth and Dynamics of Service Sector Economy}.
\end{barticle}
\endbibitem

\bibitem[\protect\citeauthoryear{Bettinelli and
  Miermont}{2017}]{bettinelli2017}
\begin{barticle}[author]
\bauthor{\bsnm{Bettinelli},~\bfnm{J{\'e}r{\'e}mie}\binits{J.}} \AND
  \bauthor{\bsnm{Miermont},~\bfnm{Gr{\'e}gory}\binits{G.}}
(\byear{2017}).
\btitle{Compact brownian surfaces I: Brownian disks}.
\bjournal{Probability Theory and Related Fields}
\bvolume{167}
\bpages{555--614}.
\end{barticle}
\endbibitem

\bibitem[\protect\citeauthoryear{Brocas and Carrillo}{2004}]{brocas2004}
\begin{barticle}[author]
\bauthor{\bsnm{Brocas},~\bfnm{Isabelle}\binits{I.}} \AND
  \bauthor{\bsnm{Carrillo},~\bfnm{Juan~D}\binits{J.~D.}}
(\byear{2004}).
\btitle{Do the “three-point victory” and “golden goal” rules make
  soccer more exciting?}
\bjournal{Journal of Sports Economics}
\bvolume{5}
\bpages{169--185}.
\end{barticle}
\endbibitem

\bibitem[\protect\citeauthoryear{Curien and Le~Gall}{2014}]{curien2014}
\begin{barticle}[author]
\bauthor{\bsnm{Curien},~\bfnm{Nicolas}\binits{N.}} \AND
  \bauthor{\bsnm{Le~Gall},~\bfnm{Jean-Fran{\c{c}}ois}\binits{J.-F.}}
(\byear{2014}).
\btitle{The brownian plane}.
\bjournal{Journal of Theoretical Probability}
\bvolume{27}
\bpages{1249--1291}.
\end{barticle}
\endbibitem

\bibitem[\protect\citeauthoryear{Dilger and Geyer}{2009}]{dilger2009}
\begin{barticle}[author]
\bauthor{\bsnm{Dilger},~\bfnm{Alexander}\binits{A.}} \AND
  \bauthor{\bsnm{Geyer},~\bfnm{Hannah}\binits{H.}}
(\byear{2009}).
\btitle{Are three points for a win really better than two? A comparison of
  German soccer league and cup games}.
\bjournal{Journal of Sports Economics}
\bvolume{10}
\bpages{305--318}.
\end{barticle}
\endbibitem

\bibitem[\protect\citeauthoryear{Falconer}{2004}]{falconer2004}
\begin{bbook}[author]
\bauthor{\bsnm{Falconer},~\bfnm{Kenneth}\binits{K.}}
(\byear{2004}).
\btitle{Fractal geometry: mathematical foundations and applications}.
\bpublisher{John Wiley \& Sons}.
\end{bbook}
\endbibitem

\bibitem[\protect\citeauthoryear{Feynman}{1949}]{feynman1949}
\begin{barticle}[author]
\bauthor{\bsnm{Feynman},~\bfnm{Richard~Phillips}\binits{R.~P.}}
(\byear{1949}).
\btitle{Space-time approach to quantum electrodynamics}.
\bjournal{Physical Review}
\bvolume{76}
\bpages{769}.
\end{barticle}
\endbibitem

\bibitem[\protect\citeauthoryear{Frick, Iijima and
  Strzalecki}{2019}]{frick2019}
\begin{barticle}[author]
\bauthor{\bsnm{Frick},~\bfnm{Mira}\binits{M.}},
  \bauthor{\bsnm{Iijima},~\bfnm{Ryota}\binits{R.}} \AND
  \bauthor{\bsnm{Strzalecki},~\bfnm{Tomasz}\binits{T.}}
(\byear{2019}).
\btitle{Dynamic random utility}.
\bjournal{Econometrica}
\bvolume{87}
\bpages{1941--2002}.
\end{barticle}
\endbibitem

\bibitem[\protect\citeauthoryear{Fujiwara}{2017}]{fujiwara2017}
\begin{bbook}[author]
\bauthor{\bsnm{Fujiwara},~\bfnm{Daisuke}\binits{D.}}
(\byear{2017}).
\btitle{Rigorous time slicing approach to Feynman path integrals}.
\bpublisher{Springer}.
\end{bbook}
\endbibitem

\bibitem[\protect\citeauthoryear{Garicano and
  Palacios-Huerta}{2005}]{garicano2005}
\begin{barticle}[author]
\bauthor{\bsnm{Garicano},~\bfnm{Luis}\binits{L.}} \AND
  \bauthor{\bsnm{Palacios-Huerta},~\bfnm{Ignacio~Isabel}\binits{I.~I.}}
(\byear{2005}).
\btitle{Sabotage in tournaments: Making the beautiful game a bit less
  beautiful}.
\end{barticle}
\endbibitem

\bibitem[\protect\citeauthoryear{Granas and Dugundji}{2003}]{granas2003}
\begin{bincollection}[author]
\bauthor{\bsnm{Granas},~\bfnm{Andrzej}\binits{A.}} \AND
  \bauthor{\bsnm{Dugundji},~\bfnm{James}\binits{J.}}
(\byear{2003}).
\btitle{Elementary Fixed Point Theorems}.
In \bbooktitle{Fixed Point Theory}
\bpages{9--84}.
\bpublisher{Springer}.
\end{bincollection}
\endbibitem

\bibitem[\protect\citeauthoryear{Guedes and Machado}{2002}]{guedes2002}
\begin{barticle}[author]
\bauthor{\bsnm{Guedes},~\bfnm{Jos{\'e}~Correia}\binits{J.~C.}} \AND
  \bauthor{\bsnm{Machado},~\bfnm{Fernando~S}\binits{F.~S.}}
(\byear{2002}).
\btitle{Changing rewards in contests: Has the three-point rule brought more
  offense to soccer?}
\bjournal{Empirical Economics}
\bvolume{27}
\bpages{607--630}.
\end{barticle}
\endbibitem

\bibitem[\protect\citeauthoryear{Gwynne and Miller}{2016}]{gwynne2016}
\begin{barticle}[author]
\bauthor{\bsnm{Gwynne},~\bfnm{Ewain}\binits{E.}} \AND
  \bauthor{\bsnm{Miller},~\bfnm{Jason}\binits{J.}}
(\byear{2016}).
\btitle{Metric gluing of Brownian and $\sqrt {8/3}$-Liouville quantum gravity
  surfaces}.
\bjournal{arXiv preprint arXiv:1608.00955}.
\end{barticle}
\endbibitem

\bibitem[\protect\citeauthoryear{Hellman and Levy}{2019}]{hellman2019}
\begin{barticle}[author]
\bauthor{\bsnm{Hellman},~\bfnm{Ziv}\binits{Z.}} \AND
  \bauthor{\bsnm{Levy},~\bfnm{Yehuda~John}\binits{Y.~J.}}
(\byear{2019}).
\btitle{Measurable selection for purely atomic games}.
\bjournal{Econometrica}
\bvolume{87}
\bpages{593--629}.
\end{barticle}
\endbibitem

\bibitem[\protect\citeauthoryear{Hossain and Ahamed}{2015}]{hossain2015}
\begin{barticle}[author]
\bauthor{\bsnm{Hossain},~\bfnm{Md~Saimum}\binits{M.~S.}} \AND
  \bauthor{\bsnm{Ahamed},~\bfnm{Faruque}\binits{F.}}
(\byear{2015}).
\btitle{Determinants of bank profitability: A study on the banking sector of
  Bangladesh}.
\bjournal{Journal of Finance and Banking}
\bvolume{13}
\bpages{43--57}.
\end{barticle}
\endbibitem

\bibitem[\protect\citeauthoryear{Hua, Polansky and Pramanik}{2019}]{hua2019}
\begin{barticle}[author]
\bauthor{\bsnm{Hua},~\bfnm{Lei}\binits{L.}},
  \bauthor{\bsnm{Polansky},~\bfnm{Alan}\binits{A.}} \AND
  \bauthor{\bsnm{Pramanik},~\bfnm{Paramahansa}\binits{P.}}
(\byear{2019}).
\btitle{Assessing bivariate tail non-exchangeable dependence}.
\bjournal{Statistics \& Probability Letters}
\bvolume{155}
\bpages{108556}.
\end{barticle}
\endbibitem

\bibitem[\protect\citeauthoryear{Islam, Alam and Chowdhury}{}]{islamrole}
\begin{barticle}[author]
\bauthor{\bsnm{Islam},~\bfnm{Md~Nazrul}\binits{M.~N.}},
  \bauthor{\bsnm{Alam},~\bfnm{Mohammad~Masud}\binits{M.~M.}} \AND
  \bauthor{\bsnm{Chowdhury},~\bfnm{Mohammad Ashraful~Ferdous}\binits{M.~A.~F.}}
\btitle{Role of Corporate Governance on Performance of Private Commercial Banks
  in Bangladesh: An Econometric Analysis}.
\end{barticle}
\endbibitem

\bibitem[\protect\citeauthoryear{Kappen}{2007}]{kappen2007}
\begin{binproceedings}[author]
\bauthor{\bsnm{Kappen},~\bfnm{Hilbert~J}\binits{H.~J.}}
(\byear{2007}).
\btitle{An introduction to stochastic control theory, path integrals and
  reinforcement learning}.
In \bbooktitle{AIP conference proceedings}
\bvolume{887}
\bpages{149--181}.
\bpublisher{AIP}.
\end{binproceedings}
\endbibitem

\bibitem[\protect\citeauthoryear{Kurtz and Swartz}{2004}]{kurtz2004}
\begin{bbook}[author]
\bauthor{\bsnm{Kurtz},~\bfnm{Douglas~S}\binits{D.~S.}} \AND
  \bauthor{\bsnm{Swartz},~\bfnm{Charles~W}\binits{C.~W.}}
(\byear{2004}).
\btitle{Theories of integration: the integrals of Riemann, Lebesgue,
  Henstock-Kurzweil, and Mcshane}
\bvolume{9}.
\bpublisher{World Scientific Publishing Company}.
\end{bbook}
\endbibitem

\bibitem[\protect\citeauthoryear{Marcet and Marimon}{2019}]{marcet2019}
\begin{barticle}[author]
\bauthor{\bsnm{Marcet},~\bfnm{Albert}\binits{A.}} \AND
  \bauthor{\bsnm{Marimon},~\bfnm{Ramon}\binits{R.}}
(\byear{2019}).
\btitle{Recursive contracts}.
\bjournal{Econometrica}
\bvolume{87}
\bpages{1589--1631}.
\end{barticle}
\endbibitem

\bibitem[\protect\citeauthoryear{Miller and Sheffield}{2015}]{miller2015}
\begin{barticle}[author]
\bauthor{\bsnm{Miller},~\bfnm{Jason}\binits{J.}} \AND
  \bauthor{\bsnm{Sheffield},~\bfnm{Scott}\binits{S.}}
(\byear{2015}).
\btitle{Liouville quantum gravity and the Brownian map I: The QLE (8/3, 0)
  metric}.
\bjournal{arXiv preprint arXiv:1507.00719}.
\end{barticle}
\endbibitem

\bibitem[\protect\citeauthoryear{Miller and Sheffield}{2016}]{miller2016}
\begin{barticle}[author]
\bauthor{\bsnm{Miller},~\bfnm{Jason}\binits{J.}} \AND
  \bauthor{\bsnm{Sheffield},~\bfnm{Scott}\binits{S.}}
(\byear{2016}).
\btitle{Liouville quantum gravity and the Brownian map III: the conformal
  structure is determined}.
\bjournal{arXiv preprint arXiv:1608.05391}.
\end{barticle}
\endbibitem

\bibitem[\protect\citeauthoryear{Minar}{}]{minarrevisiting}
\begin{barticle}[author]
\bauthor{\bsnm{Minar},~\bfnm{Sarwar~J}\binits{S.~J.}}
\btitle{Revisiting theEvolution ofInternational System: ReflectionsontheRole
  ofReligion}.
\end{barticle}
\endbibitem

\bibitem[\protect\citeauthoryear{Minar}{2018}]{minar2018}
\begin{barticle}[author]
\bauthor{\bsnm{Minar},~\bfnm{Sarwar~J}\binits{S.~J.}}
(\byear{2018}).
\btitle{Grand Strategy and Foreign Policy: How Grand Strategy Can Aid
  Bangladesh's Foreign Policy Rethinking?}
\bjournal{Journal of Social Studies}
\bvolume{4}
\bpages{20--27}.
\end{barticle}
\endbibitem

\bibitem[\protect\citeauthoryear{Minar}{2019}]{minar2019}
\begin{barticle}[author]
\bauthor{\bsnm{Minar},~\bfnm{Sarwar~J}\binits{S.~J.}}
(\byear{2019}).
\btitle{Tatmadaw’s Crackdown on The Rohingyas: A SWOT Analysis}.
\bjournal{Journal of Social Studies}
\bvolume{5}
\bpages{1--5}.
\end{barticle}
\endbibitem

\bibitem[\protect\citeauthoryear{Minar and Halim}{2021}]{minar2021}
\begin{barticle}[author]
\bauthor{\bsnm{Minar},~\bfnm{Sarwar~J}\binits{S.~J.}} \AND
  \bauthor{\bsnm{Halim},~\bfnm{Abdul}\binits{A.}}
(\byear{2021}).
\btitle{The Rohingyas of Rakhine State: Social Evolution and History in the
  Light of Ethnic Nationalism}.
\bjournal{arXiv preprint arXiv:2106.02945}.
\end{barticle}
\endbibitem

\bibitem[\protect\citeauthoryear{Mohammad and Mohammad}{2010}]{mohammad2010}
\begin{barticle}[author]
\bauthor{\bsnm{Mohammad},~\bfnm{Masud~Alam}\binits{M.~A.}} \AND
  \bauthor{\bsnm{Mohammad},~\bfnm{Rafiqul~Islam}\binits{R.~I.}}
(\byear{2010}).
\btitle{Revisiting the Feldstein-Horioka Hypothesis of savings, investment and
  capital mobility: evidence from 27 EU countries}.
\end{barticle}
\endbibitem

\bibitem[\protect\citeauthoryear{Moschini}{2010}]{moschini2010}
\begin{barticle}[author]
\bauthor{\bsnm{Moschini},~\bfnm{GianCarlo}\binits{G.}}
(\byear{2010}).
\btitle{Incentives And Outcomes In A Strategic Setting: The 3-Points-For-A-Win
  System In Soccer}.
\bjournal{Economic Inquiry}
\bvolume{48}
\bpages{65--79}.
\end{barticle}
\endbibitem

\bibitem[\protect\citeauthoryear{Muldowney}{2012}]{muldowney2012}
\begin{bbook}[author]
\bauthor{\bsnm{Muldowney},~\bfnm{Pat}\binits{P.}}
(\byear{2012}).
\btitle{A modern theory of random variation}.
\bpublisher{Wiley Online Library}.
\end{bbook}
\endbibitem

\bibitem[\protect\citeauthoryear{Parthasarathy, Theorems and
  Applications}{1972}]{parthasarathy1972}
\begin{bmisc}[author]
\bauthor{\bsnm{Parthasarathy},~\bfnm{T}\binits{T.}},
  \bauthor{\bsnm{Theorems},~\bfnm{Selection}\binits{S.}} \AND
  \bauthor{\bsnm{Applications},~\bfnm{Their}\binits{T.}}
(\byear{1972}).
\btitle{Lectures Notes in Math, nr. 263}.
\end{bmisc}
\endbibitem

\bibitem[\protect\citeauthoryear{Polansky and Pramanik}{2021}]{polansky2021}
\begin{barticle}[author]
\bauthor{\bsnm{Polansky},~\bfnm{Alan~M}\binits{A.~M.}} \AND
  \bauthor{\bsnm{Pramanik},~\bfnm{Paramahansa}\binits{P.}}
(\byear{2021}).
\btitle{A motif building process for simulating random networks}.
\bjournal{Computational Statistics \& Data Analysis}
\bvolume{162}
\bpages{107263}.
\end{barticle}
\endbibitem

\bibitem[\protect\citeauthoryear{Pramanik}{2016}]{pramanik2016}
\begin{bbook}[author]
\bauthor{\bsnm{Pramanik},~\bfnm{Paramahansa}\binits{P.}}
(\byear{2016}).
\btitle{Tail non-exchangeability}.
\bpublisher{Northern Illinois University}.
\end{bbook}
\endbibitem

\bibitem[\protect\citeauthoryear{Pramanik}{2020}]{pramanik2020}
\begin{binproceedings}[author]
\bauthor{\bsnm{Pramanik},~\bfnm{Paramahansa}\binits{P.}}
(\byear{2020}).
\btitle{Optimization of market stochastic dynamics}.
In \bbooktitle{SN Operations Research Forum}
\bvolume{1}
\bpages{1--17}.
\bpublisher{Springer}.
\end{binproceedings}
\endbibitem

\bibitem[\protect\citeauthoryear{Pramanik}{2021a}]{pramanik2021consensus}
\begin{barticle}[author]
\bauthor{\bsnm{Pramanik},~\bfnm{Paramahansa}\binits{P.}}
(\byear{2021}a).
\btitle{Consensus as a Nash Equilibrium of a stochastic differential game}.
\bjournal{arXiv preprint arXiv:2107.05183}.
\end{barticle}
\endbibitem

\bibitem[\protect\citeauthoryear{Pramanik}{2021b}]{pramanik2021effects}
\begin{barticle}[author]
\bauthor{\bsnm{Pramanik},~\bfnm{Paramahansa}\binits{P.}}
(\byear{2021}b).
\btitle{Effects of water currents on fish migration through a Feynman-type path
  integral approach under $\sqrt{8/3}$ Liouville-like quantum gravity
  surfaces}.
\bjournal{Theory in Biosciences}
\bvolume{140}
\bpages{205--223}.
\end{barticle}
\endbibitem

\bibitem[\protect\citeauthoryear{Pramanik}{2021c}]{pramanik2021s}
\begin{barticle}[author]
\bauthor{\bsnm{Pramanik},~\bfnm{Paramahansa}\binits{P.}}
(\byear{2021}c).
\btitle{Effects of water currents on fish migration through a Feynman-type path
  integral approach under [Formula: see text] Liouville-like quantum gravity
  surfaces.}
\bjournal{Theory in Biosciences= Theorie in den Biowissenschaften}.
\end{barticle}
\endbibitem

\bibitem[\protect\citeauthoryear{Pramanik}{2021d}]{pramanik2021e}
\begin{barticle}[author]
\bauthor{\bsnm{Pramanik},~\bfnm{Paramahansa}\binits{P.}}
(\byear{2021}d).
\btitle{Effects of water currents on fish migration through a Feynman-type path
  integral approach under $\sqrt{8/3}$ 8/3 Liouville-like quantum gravity
  surfaces}.
\bjournal{Theory in Biosciences}
\bvolume{140}
\bpages{205--223}.
\end{barticle}
\endbibitem

\bibitem[\protect\citeauthoryear{Pramanik and Polansky}{2019}]{pramanik2019}
\begin{barticle}[author]
\bauthor{\bsnm{Pramanik},~\bfnm{Paramahansa}\binits{P.}} \AND
  \bauthor{\bsnm{Polansky},~\bfnm{Alan~M}\binits{A.~M.}}
(\byear{2019}).
\btitle{Semicooperation under curved strategy spacetime}.
\bjournal{arXiv preprint arXiv:1912.12146}.
\end{barticle}
\endbibitem

\bibitem[\protect\citeauthoryear{Pramanik and
  Polansky}{2020a}]{pramanik2020opt}
\begin{barticle}[author]
\bauthor{\bsnm{Pramanik},~\bfnm{Paramahansa}\binits{P.}} \AND
  \bauthor{\bsnm{Polansky},~\bfnm{Alan~M}\binits{A.~M.}}
(\byear{2020}a).
\btitle{Optimization of a Dynamic Profit Function using Euclidean Path
  Integral}.
\bjournal{arXiv preprint arXiv:2002.09394}.
\end{barticle}
\endbibitem

\bibitem[\protect\citeauthoryear{Pramanik and
  Polansky}{2020b}]{pramanik2020motivation}
\begin{barticle}[author]
\bauthor{\bsnm{Pramanik},~\bfnm{Paramahansa}\binits{P.}} \AND
  \bauthor{\bsnm{Polansky},~\bfnm{Alan~M}\binits{A.~M.}}
(\byear{2020}b).
\btitle{Motivation to Run in One-Day Cricket}.
\bjournal{arXiv preprint arXiv:2001.11099}.
\end{barticle}
\endbibitem

\bibitem[\protect\citeauthoryear{Pramanik and Polansky}{2021}]{pramanik2021}
\begin{barticle}[author]
\bauthor{\bsnm{Pramanik},~\bfnm{Paramahansa}\binits{P.}} \AND
  \bauthor{\bsnm{Polansky},~\bfnm{Alan~M}\binits{A.~M.}}
(\byear{2021}).
\btitle{Optimal Estimation of Brownian Penalized Regression Coefficients}.
\bjournal{arXiv preprint arXiv:2107.02291}.
\end{barticle}
\endbibitem

\bibitem[\protect\citeauthoryear{Ross}{2008}]{ross2008}
\begin{barticle}[author]
\bauthor{\bsnm{Ross},~\bfnm{Kevin}\binits{K.}}
(\byear{2008}).
\btitle{Stochastic control in continuous time}.
\bjournal{Lecture Notes on Continuous Time Stochastic Control, Spring}.
\end{barticle}
\endbibitem

\bibitem[\protect\citeauthoryear{Santos}{2014}]{santos2014}
\begin{barticle}[author]
\bauthor{\bsnm{Santos},~\bfnm{Ricardo~Manuel}\binits{R.~M.}}
(\byear{2014}).
\btitle{Optimal Soccer Strategies}.
\bjournal{Economic Inquiry}
\bvolume{52}
\bpages{183--200}.
\end{barticle}
\endbibitem

\end{thebibliography}
\end{document}